\documentclass[12pt,reqno,english]{amsart}
\usepackage{fullpage}
\usepackage{amsfonts}
\usepackage{amssymb}
\usepackage[T1]{fontenc}
\usepackage[utf8]{inputenc}
\usepackage{babel}
\usepackage{stmaryrd}
\usepackage{comment}
\usepackage{xcolor}
\usepackage[inline]{enumitem}

\usepackage{csquotes}
\usepackage{newtx} 
\usepackage{amsmath}
\usepackage{url}
\usepackage{tikz-cd}
\usepackage{mathtools}
\usepackage{multicol}
\usepackage{graphicx}
\usepackage{amsmath,calligra,mathrsfs}
\usepackage{xurl}
\usepackage[maxbibnames=10,style=ext-alphabetic,backend=bibtex,url=false,eprint=false,isbn=false,articlein=false]{biblatex}
\usepackage%
 [pdfusetitle,
 bookmarksdepth=subsection,
 unicode]%
 {hyperref}

\numberwithin{equation}{section}

\newtheorem{Theorem}[equation]{Theorem}
\newtheorem{Theoremx}{Theorem} %
	
\newtheorem{Corollary}[equation]{Corollary}
\newtheorem{Lemma}[equation]{Lemma}
\newtheorem{Proposition}[equation]{Proposition}

\theoremstyle{definition}
\newtheorem{Example}[equation]{Example}
\newtheorem{Definition}[equation]{Definition}

\newtheorem{Conjecture}[equation]{Conjecture}
\newtheorem{Construction}[equation]{Construction}
\newtheorem{Remark}[equation]{Remark}

\newtheorem{Question}[equation]{Question}

    \newtheorem{Notation}[equation]{Notation}
  \newtheorem{Warning}[equation]{Warning}

\newcommand{\SL}{\textnormal{SL}}
\newcommand{\Sc}{\textnormal{sc}}
\newcommand{\psc}{\textnormal{(sc)}}

\newcommand{\Inn}{\textnormal{Inn}}
\newcommand{\ov}{\overline}
\newcommand{\wt}{\widetilde}
\newcommand\restr[2]{{
  \left.\kern-\nulldelimiterspace %
  #1 %
  \vphantom{\big|} 
  \right|_{#2} 
  }}

\newcommand{\Hom}{\textnormal{Hom}}
\DeclareMathOperator\Irr{Irr}
\newcommand*{\sheafhom}{\textnormal{H}\kern -.5pt om}

\newcommand{\Res}{\textnormal{Res}}
\DeclareMathOperator{\sHom}{\mathscr{H}\text{\kern -3pt {\calligra\large om}}\,}
\newcommand{\Z}{\mathbb{Z}}
\newcommand{\Q}{\mathbb{Q}}

\newcommand{\cE}{\mc{E}}
\newcommand{\cC}{\mc{C}}
\newcommand{\GL}{\textnormal{GL}}
\newcommand{\C}{\mathbb{C}}

\newcommand\Aut{\textnormal{Aut}}
\newcommand\Gal{\textnormal{Gal}}
\newcommand\PGL{\textnormal{PGL}}
\newcommand\cal{\mathcal}
\newcommand\bbC{\mathbb{C}}
\newcommand\bbD{\mathbb{D}}
\newcommand\bbF{\mathbb{F}}
\newcommand\bbG{\mathbb{G}}

\newcommand\bbP{\mathbb{P}}
\newcommand\bbQ{\mathbb{Q}}
\newcommand\bbZ{\mathbb{Z}}
\newcommand\scrB{\mathscr{B}}
\newcommand\scrM{\mathscr{M}}
\newcommand\scrT{\mathscr{T}}

\newcommand\tn{\textnormal}

\newcommand{\mc}{\mathscr} %

\newcommand\onto\twoheadrightarrow
\newcommand\into\hookrightarrow
\newcommand\longisoto{\mathrel{\tilde\longrightarrow}}
\newcommand\isoto{\mathrel{\tilde\to}}

\newcommand\smat[4]{\big(\begin{smallmatrix}#1&#2\\#3&#4\end{smallmatrix}\big)}

\newcommand\Greg{{G\textnormal{-reg}}}
\newcommand\Gphireg{{(G,\phi)\textnormal{-reg}}}
\newcommand\actson\curvearrowright

\newcommand\defeq{:=}

\begin{document}

\title[NBrigv22.7.26]{Non-basic rigid packets for discrete $L$-parameters}

\author{Peter Dillery}
\address{University of Bonn, Mathematics Institute}
\email{dillery@math.uni-bonn.de}
\thanks{The first author was supported by the Brin postdoctoral fellowship}

\author{David Schwein}
\address{University of Utah, Department of Mathematics}
\email{david.schwein@utah.edu}

\begin{abstract}
This article initiates the study of non-basic rigid inner forms over $p$-adic local fields,
extending the basic theory developed by Kaletha.
Motivated by the recent work of Bertoloni Meli--Oi
on the $B(G)$-parametrization of the local Langlands conjectures,
our main application is to extend the basic rigid refined local Langlands conjectures
for a discrete $L$-parameter $\phi$ of a quasi-split connected reductive group $G$.
The packets of our extended construction are Weyl orbits of representations
of inner forms of twisted Levi subgroups $N$ of $G$ for which $\phi$
factors through a member of the canonical $L$-embeddings
$^{L}N_{\pm} \to \prescript{L}{}G$ constructed by Kaletha.
\end{abstract}

\maketitle

\tableofcontents

\section{Introduction}

Let $G$ be a reductive group over a nonarchimedean local field~$F$.
Elliptic twisted Levi subgroups of $G$
are a rich source of discrete $L$-parameters~$\phi$ of~$G$
and discrete series representations of~$G(F)$,
objects that are predicted to match under the local Langlands conjectures.
Perhaps the most prominent example of this phenomenon is
Yu's construction of supercuspidal representations \cite{Yu01}
and Kaletha's parametrization of most of them \cite{Kaletha21}.
The goal of this paper is to show that the rigid refined local Langlands
conjectures imply a new, extended correspondence
that organizes representations of all twisted Levi subgroups $N$ of $G$
that yield $\phi$ by generalized Langlands functoriality,
that is, for which $\phi$ factors through an $L$-embedding $^{L}N_\pm \to \prescript{L}{}G$.
Our main technical accomplishment 
is to develop analogues of the Newton and Kottwitz maps on $B(G)$
for the set $H^1_\tn{reg}(\cE_\tn{rig},G)$ of regular cohomology classes of the rigid gerbe.

\subsection{Background}
The local Langlands conjectures predict a finite-to-one map
from irreducible (smooth) representations of $G(F)$
on $\C$-vector spaces to $\widehat{G}$-conjugacy classes of $L$-parameters
$\phi\colon W_{F} \times \SL_{2} \to \prescript{L}{}G \defeq \widehat{G} \rtimes W_{F}$,
where $\widehat{G}$ is a Langlands dual group of $G$ and $W_{F}$ is the Weil group of $F$.
The conjectures further predict this map to satisfy numerous desiderata
(see for instance~\cite[\S 6.1]{Taibi22}). 

A refined local Langlands correspondence is a parametrization
of the fibers of this finite-to-one map using precise auxiliary data associated to $\phi$.
When $G$ is quasi-split, the conjectural parametrization asserts that
the fiber over the conjugacy class of an $L$-parameter~$\phi$,
denoted by $\Pi_{\phi}(G)$,
is in bijection with irreducible representations of the finite group
$\pi_{0}(Z_{\widehat{G}}(\phi)/Z(\widehat{G})^{\Gamma})$.
When $G$ is not quasi-split,
Adams, Barbasch, and Vogan discovered, originally for $F=\mathbb{R}$ (see~\cite{ABV92}),
that one can group together representations of inner twists of $G$,
which all have the same set of $L$-parameters,
and they gave a precise parametrization of this \textit{compound $L$-packet}.
Following their philosophy, in the rest of the introduction we take $G$ to be quasi-split.

A subtle question in the study of compound $L$-packets
is when to identify representations of two inner twists
$(G_{1}, \psi_{1}, \pi_{1})$ and $(G_{2}, \psi_{2}, \pi_{2})$,
where $\psi_i\colon G_{\ov{F}} \to G_{i,\ov{F}}$
is a choice of twisting isomorphism and $\pi_{i}$ is a representation of $G_{i}(F)$.
As Vogan observed in \cite{Vogan93},
the natural equivalence of inner twists
is unsatisfactory because there are automorphisms of inner twists
which identify non-isomorphic representations.
To overcome this obstacle, Vogan chose a $1$-cocycle $z_{\psi}$
valued in $G$ and lifting $\bar{z}_{\psi} \in Z^{1}(F, G_{\textnormal{ad}})$,
the cocycle associated to the inner twisting $\psi$,
and worked with triples $(G', z_{\psi}, \pi)$ instead,
insisting that isomorphisms of such triples preserve the cocycle up to a coboundary.
This correction is not entirely satisfactory either
because a lift $z_\psi$ need not exist for every inner form,
meaning that not all reductive groups can be reached.

To rectify these shortcomings, it has proved useful to take $z_{\psi}$
to be a cocycle not for the absolute Galois group~$\Gamma$ of~$F$
but for a certain more general type of group~$\mc E$
called by Langlands and Rapoport a \textit{Galois gerbe} \cite[\S~2]{LanglandsRapoport87}.
Such a group~$\mc E$ is, by definition, an extension
\[
1 \to A(\overline F) \to \mc E \to \Gamma \to 1
\]
of~$\Gamma$ by a commutative group scheme~$A$,
called the \textit{band}.
Then $\mc E$ acts on $G(\ov{F})$ through the inflation of the Galois action,
allowing one to form the nonabelian group cohomology set
$H^1(\cE,G)\defeq H^1(\cE,G(\ov{F}))$.
There are three standard choices for~$\mc{E}$,
each yielding a flavor of the local Langlands conjectures:
the trivial Galois gerbe $\Gamma$, yielding pure inner forms;
the isocrystal gerbe $\mc E_\tn{iso}$, yielding extended pure inner forms;
and the rigid gerbe $\mc E_\tn{rig}$, yielding rigid inner forms.
See \cite[\S~6.3]{Taibi22} for a discussion of all three of these approaches.
The third of these will be the focus of this paper,
but first we discuss the isocrystal gerbe.

\subsubsection*{The algebraic structure of $B(G)$}
Suppose $G$ is quasi-split.
When $\mc{E} = \mc{E}_\tn{iso}$ is the isocrystal gerbe,
the band is the pro-torus $\bbD$ with character group~$\bbQ$.
In other words, $\mathbb D\defeq\varprojlim_n \bbG_m$
with the inverse limit taken over nonnegative integers with transition map for
$n\mid m$ the $(m/n)$th-power map.
Define $B(G)$ to be the classes in $H^1(\mc E_\tn{iso},G)$
represented by a cocycle $b$ whose restriction $b|_\bbD$ to $\bbD(\ov{F})$ is algebraic,
and define the set of basic classes
$B(G)_{\tn{bas}}\subseteq B(G)$ to be those classes whose image in
$H^{1}(\cE_{\tn{iso}}, G_{\tn{ad}})$ is contained in $H^{1}(F, G_{\tn{ad}})$.

Given an (algebraic) cocycle $b\in Z^1(\mc E_\tn{iso},G)$
with $b|_\bbD$ is defined over $F$,
we can form the centralizer $Z_G(b|_\bbD)$,
a Levi subgroup of~$G$ whose stable conjugacy class
is independent of the choice of $b$ within its cohomology class.
Refining this construction,
we can twist the quasi-split inner form of $Z_G(b|_\bbD)$
by the image of $b$ in $H^1(F,G_\tn{ad})$,
forming the \textit{extended pure inner form}~$G_b$.
If $b$ is basic, or equivalently, if $b|_\bbD$ factors through $Z(G)$,
then $G_b$ is an inner twist of~$G$. 

One also has the Newton and Kottwitz maps
\begin{equation}\label{bgnewtkott}
\nu\colon B(G) \to \pi_1(G)_\Gamma,\qquad
\kappa\colon B(G)\to (X_*(T)_\bbQ^+)^\Gamma,
\end{equation}
where $T$ is a minimal Levi subgroup of~$G$.

\subsubsection*{The $B(G)$ local Langlands conjectures}
Kottwitz has conjectured \cite[Conjecture~4.1]{Kaletha14} that there is a bijection
\begin{equation}\label{BGparam}
\iota_\mathfrak{w} \colon
\bigsqcup_{b \in B(G)_{\tn{bas}}} \Pi_{\phi}(G_{b}) \longrightarrow
\Irr(Z_{\widehat{G}}(\phi)/Z_{\widehat{G}_{\tn{der}}}^\circ(\phi)),
\end{equation}
depending on a choice of Whittaker datum $\mathfrak{w}$ for~$G$.
When the center of $G$ is connected,
so that the map $B(G)_{\tn{bas}} \to H^{1}(F, G_{\tn{ad}})$
is surjective, the bijection parametrizes $\Pi_\phi(G')$
for any inner twist $G'$ of~$G$.
In fact, there are many such parametrizations,
since different elements $b \in B(G)_{\tn{bas}}$
can yield the same inner form $G_{b}$. 

The sets $B(G)$ and $B(G)_{\tn{bas}}$ also appear in the geometry of the space $\textnormal{Bun}_{G}$, the moduli stack of $G$-bundles on the Fargues--Fontaine curve, which underlies Fargues and Scholze's geometrization of the local Langlands correspondence \cite{FS21}. Namely, the topological space $|\textnormal{Bun}_G|$ is homeomorphic to $B(G)$ and contains $B(G)_{\tn{bas}}$ as the semistable locus. In light of this description and the role of $\textnormal{Bun}_G$ in the local Langlands correspondence, it is natural to seek to extend the conjectural bijection \eqref{BGparam} to all of $B(G)$.

This extension was carried out by Bertoloni Meli and Oi in \cite[Theorem 1.1]{BMOi23}. Given a basic local Langlands correspondence \eqref{BGparam} for $G$ and its Levi subgroups, they define for every $L$-parameter $\phi$ and every $b\in B(G)$ an $L$-packet $\Pi_\phi(G_b)$ and construct a bijection
\begin{equation}\label{BMObijection}
\bigsqcup_{b \in B(G)} \Pi_{\phi}(G_{b}) \xrightarrow{\iota_{\mathfrak{w}}} \Irr(S_\phi)
\qquad (S_\phi\defeq Z_{\widehat{G}}(\phi)).
\end{equation}
The extended parametrization has many advantages. For example, as explained in the introduction to \cite{BMOi23}, irreducible representations of $S_{\phi}$ correspond to coherent sheaves on the classifying stack $[\ast/S_{\phi}]$, which naturally embeds into $Z^{1}(W_{F}, \widehat{G})/\widehat{G}$, the stack of $L$-parameters of~$G$. As a result one expects the extended parametrization to provide a rough framework for relating the construction of Fargues and Scholze to the refined local Langlands conjectures. The Bertoloni Meli--Oi framework also allows for a generalization, \cite[Theorem 5.13]{BMOi23}, of the endoscopic character identities involving Levi subgroups of endoscopic groups.

\subsubsection*{Searching for twisted Levi subgroups}
When $\phi$ is a discrete $L$-parameter,
the Bertoloni Meli--Oi parametrization reduces
to the original basic conjecture~\eqref{BGparam}:
by definition of discreteness, $\phi$ does not factor through the
$L$-embedding associated to any proper Levi subgroup of~$G$.
For the study of discrete parameters,
it is much more convenient to work with \textit{twisted} Levi subgroups of~$G$,
those $F$-rational subgroups that become isomorphic to a Levi subgroup
of~$G_{\overline F}$ over~$\overline F$.

If $M \subseteq G$ is a twisted Levi subgroup then there is in general
no canonical embedding $^{L}M \to \prescript{L}{}G$.
To circumvent this difficulty, Kaletha constructed in \cite{Kaletha19b}
a double cover $M(F)_{\pm} \to M(F)$ and an $L$-group $^{L}M_{\pm}$
which does admit a canonical $\widehat{G}$-conjugacy class of embeddings
$^{L}M_{\pm} \to \prescript{L}{}G$.
For brevity we will refer to any such $L$-embedding as \textit{admissible}.
Although $\phi$ cannot factor through the $L$-embedding for a proper Levi subgroup,
it can, and often does, factor through the $L$-embedding
of a proper twisted Levi subgroup.
For example, when $p$ does not divide the order of the Weyl group of~$G$,
every semisimple discrete $L$-parameter factors through an admissible $L$-embedding
for an elliptic maximal torus.
These factorizations are key to Kaletha's construction
of torally wild supercuspidal $L$-packets \cite{Kaletha21}.

These developments, and the analogy with \cite{BMOi23}, lead us to ask:

\begin{Question} \label{mainQuestion}
Is there a version of the refined local Langlands conjectures which constructs, for every twisted Levi subgroup $N$ of $G$ such that $\phi$ factors through an admissible $L$-embedding $^{L}N_{\pm} \to \prescript{L}{}G$, a packet of representations of rigid inner forms of $N$? 
\end{Question}

\subsection{Results}
One goal of this paper is to propose a coarse answer to Question~\ref{mainQuestion}.
Our main object of study and tool towards this end is the cohomology set
$H^1_\tn{reg}(\cE_\tn{rig},G)$, which we will show has many formal similarities to $B(G)$
but seems better adapted to the study of discrete $L$-parameters.
Assume again that $G$ is quasi-split.

\subsubsection*{Non-basic classes for $\mc{E}_{\tn{rig}}$}
For $\mc E_\tn{rig}$, the band is the group
$u \defeq \varprojlim_{E/F,n} \Res_{E/F}(\mu_n)$,
where the projective limit is taken over finite Galois extensions $E/F$ and positive integers~$n$ and where the transition map for $(K/E,n\mid m)$ is the $(m/n)$th-power of the $K/E$ norm map. We can again define $H^1_\tn{alg}(\mc{E}_\tn{rig},G)$ as the classes whose restriction to the band is algebraic, and define the basic subset $H^1_\tn{bas}(\mc{E}_\tn{rig},G)$ as we defined $B(G)_\tn{bas}$, replacing $\mc{E}_\tn{iso}$ with $\mc{E}_\tn{rig}$. In this setting, the map $H^{1}_{\tn{bas}}(\mc{E}_{\tn{rig}}, G) \to H^{1}(F, G_{\tn{ad}})$ is always surjective.

In the same way that $b\in B(G)$ gave rise to a reductive $F$-group $G_b$, a class $[x]\in H^1_\tn{alg}(\mc{E}_\tn{rig},G)$ gives rise to a reductive $F$-group $G_{[x]}$: choosing a representative~$x$ of $[x]$ whose restriction to the band $x|_u\colon u_{\ov F}\to G_{\ov F}$ is defined over~$F$, we form the connected centralizer $Z_G^\circ(x|_u)$, then use $x$, which is automatically valued in $Z_{G}(x|_u)$, to construct the inner twist $G_{[x]}$ of the quasi-split inner form of $Z_G^\circ(x|_u)$.

One surprising feature of this construction is that there are strange classes $[x]\in H^1_\tn{alg}(\mc{E},G)$ for which $G_{[x]}$ has smaller rank than~$G$, or is even a finite $F$-group. For example, in Lemma~\ref{nontoralcocycle}, we construct an algebraic cocycle~$x$ of~$\mc{E}$ for $\PGL_2$ such that the image of $x|_u$ is a form of $\mu_2^2$, forcing $Z_{G_{\ov{F}}}(x|_u) = \mu_{2}^{2}$. The source of these examples is the potential failure of $x|_u$ to factor through a maximal torus. To avoid this pathology, we restrict our attention to the \textit{regular classes} in $H_{\tn{alg}}^{1}(\mc{E}_{\tn{rig}}, G)$, those lying in the image of $H^{1}_{\tn{alg}}(\mc{E}_{\tn{rig}}, T)$ for some maximal torus $T$. In particular, if $[x]$ is regular then the map $x|_u$ factors through a torus and we can use this torus to analyze the structure of~$x$ in terms of the usual structure theory of reductive groups.

Even after we restrict to the regular classes $H^1_\tn{reg}(\mc{E}_\tn{rig},G)$, there remains a rich supply of groups of the form $Z_G^\circ(x|_u)$, which we call \textit{rigid Newton centralizers}. It is difficult to combinatorially describe these groups, but, as we show in Theorem~\ref{twistedLevithm}, they include all elliptic twisted Levi subgroups of~$G$. This observation gives strong evidence that $\mc{E}_\tn{rig}$ is equipped to answer Question~\ref{mainQuestion}: unlike $\mc{E}_\tn{iso}$, it can produce subgroups of~$G$ whose $L$-embeddings are amenable to the study of discrete parameters.

\subsubsection*{The basic rigid Langlands conjectures}
The rigid analogue of the $B(G)_\tn{bas}$ local Langlands conjecture~\eqref{BGparam}, articulated in \cite[(1.3)]{Kaletha16a}, is a conjectural bijection
\begin{equation} \label{basicRigidBijection}
\iota_{\mathfrak{w}} \colon
\bigsqcup_{[x] \in H^{1}_{\tn{bas}}(\mc{E}_{\tn{rig}}, G)} \Pi_{\phi}(G_{[x]})
\longisoto \Irr(\pi_{0}(S_{\phi}^{+})),
\end{equation}
where $S_{\phi}^{+}$ is the preimage of $S_{\phi}$ in the algebraic universal cover
$\widehat G_\tn{univ}$ of~$\widehat G$ (see Section~\ref{sec:univ}).
To work towards answering Question~\ref{mainQuestion},
we would like to extend the scope of \eqref{basicRigidBijection}
by weakening ``basic'' to ``regular'',
thereby parametrizing representations of subgroups of the form $G_{[x]}$
for $[x]\in \smash{H^1_\tn{reg}(\cE_\tn{rig},G)}$.
Since not every rigid Newton centralizer is a twisted Levi subgroup
(Example~\ref{nonadmex}),
and it is unclear how the dual groups of such subgroups should embed in~$\widehat G$,
we restrict our attention to the classes $[x] \in H^{1}_{\tn{reg}}(\mc{E}, G)$ such that
$\smash{Z_{G_{\ov{F}}}^\circ(x|_u)}$ is a Levi subgroup of $G_{\ov{F}}$,
which we call \textit{Levi regular} and denote by $\smash{H^{1}_{\tn{L-reg}}(\mc{E}, G)}$.

\subsubsection*{Newton and Kottwitz maps}
One initial observation is that the assignment $x\mapsto Z_G^\circ(x|_u)$
can be used to partition $\smash{H^1_\tn{reg}(\cE_\tn{rig},G)}$
into subsets parametrized by stable conjugacy classes of rigid Newton centralizers,
a partition that we call the \textit{Newton decomposition}:
\[
H^1_\tn{reg}(\mc{E},G) = \bigsqcup_{[M]} H^1_\tn{reg}(\mc{E},G)_M.
\]
Such partitions are a motif throughout the structure theory of $H^1_\tn{reg}(\cE_\tn{rig},G)$,
and reappear in the development of the Newton and Kottwitz maps.

The Newton map $\nu\colon H^1_\tn{reg}(\cE_\tn{rig},G)\to\cC_\tn{st}(G)$
is simply restriction to the band~$u$.
The challenge is to suitably interpret the target,
a certain space of conjugacy classes of homomorphisms $u\to G$
which has a rather different flavor than
the target $(X_*(T)_\bbQ)^\Gamma$ of the Newton map for~$B(G)$.
The difference is due to the fact that $\bbD$
is connected while $u$ is totally disconnected,
and while $(X_*(T)_\bbQ)^\Gamma$ is very nearly a Weyl chamber,
the set $\cC_\tn{st}(G)$ is roughly built from affine Weyl chambers,
indexed by elliptic maximal tori of~$G$
and glued together in a certain complicated way.

The Kottwitz map is much more subtle.
Our starting point is Kaletha's functor
$\ov{Y}_{+,\tn{tor}}(Z\to G)$
and pair of isomorphisms (\cite[Theorem 4.11]{Kaletha16a})
\begin{equation} \label{kalethaisom}
H^{1}_{\tn{bas}}(\mc{E}_{\tn{rig}}, G)
\isoto \ov{Y}_{+,\tn{tor}}(Z(G)\to G)
\isoto X^{*}(\pi_{0}(Z(\widehat{G})^{\Gamma,+})),
\end{equation}
the latter of which provides the ultimate connection
to the Galois side of the local Langlands correspondence.
Conjecturally, \eqref{kalethaisom} identifies which rigid inner twist of $G$
carries the representation corresponding to a fixed
$\rho \in \Irr(\pi_{0}(S_{\phi}^{+}))$,
by passing the restriction of $\rho$ to $\pi_{0}(Z(\widehat{G})^{\Gamma,+})$
through the isomorphism.
A prerequisite for any kind of reasonable non-basic rigid correspondence
is therefore the extension of this duality map to all of
$H^{1}_{\tn{L-reg}}(\mc{E}_{\tn{rig}}, G)$. 
This extension is accomplished in Sections \ref{Kottmapsec}
and \ref{dualEsec},
culminating in \eqref{fullTN}.

\begin{Theoremx}
Kaletha's functor $\ov{Y}_{+,\tn{tor}}(Z\to G)$
extends to a functor $\ov{Y}_{+,\tn{tor}}(G)$,
and there are injective Kottwitz maps
$H^1_\tn{reg}(\cE_\tn{rig},G)\into\ov{Y}_{+,\tn{tor}}(G)$
and
\[
\begin{tikzcd}
H^1_\tn{L-reg}(\cE_\tn{rig},G) \rar[hookrightarrow] &
\displaystyle\bigsqcup_{[N]}
\frac{X^{*}(\pi_{0}(Z(\widehat{G})^{\Gamma,+}_{(\widehat N)}))}{W(G,N)(F)},
\end{tikzcd}
\]
where $[N]$ ranges over stable conjugacy classes of rigid Newton centralizers and
$Z(\widehat G)^{\Gamma,+}_{(\widehat N)}$ is the preimage of
$Z(\widehat G)^\Gamma$ in $\widehat G_\tn{univ}$.
\end{Theoremx}

\subsubsection*{Applications to the Langlands correspondence}
Using our structure theory of $H^1_\tn{reg}(\cE_\tn{rig},G)$,
we can formulate a coarse generalization of
the $B(G)$ local Langlands conjectures to the rigid setting.
Fix a discrete $L$-parameter~$\phi$.
Given a quasi-split twisted Levi subgroup $N$ of~$G$,
when $\phi$ factors through an $L$-embedding
${}^LN_\pm\into{}^LG$,
or equivalently (Lemma~\ref{lem:dualizetwistedlevis}),
normalizes the image of~$\widehat N$ under such an embedding,
there is a resulting $L$-parameter~$\phi_{N,\pm}$ for~$N$.
Assuming the basic rigid local Langlands correspondence for~$N$
as in \eqref{basicRigidBijection},
from which one can deduce a similar correspondence for $N(F)_\pm$,
we construct a compound $L$-packet $\Pi^{\cE_\tn{rig}}_\phi(N_\pm)$
of representations of $N(F)_\pm$ and its rigid inner forms.
Writing $\Phi^{\cE_\tn{rig}}_\phi(G)$
for the collection of $\widehat G$-conjugacy classes
of basic rigid enhancements of~$\phi_{N,\pm}$
for all possible~$N$ whose restriction to $Z(\widehat{G})^{\Gamma,\circ,+}_{(\widehat{N})} \hookrightarrow \Hom_{F}(u, Z(N))$ (cf. Proposition \ref{Newtonbandprop}) has connected centralizer in $G$ equal to $N$,
we can formulate the following non-basic analogue of 
\eqref{basicRigidBijection}.

\begin{Theoremx}[Theorem~\ref{maintheorem2}]
Let $\cE=\cE_\tn{rig}$
and let $\phi$ be a discrete $L$-parameter for a quasi-split group~$G$.
Assume that the basic rigid local Langlands correspondence
\eqref{basicRigidBijection} exists for $G$
and each of its quasi-split elliptic twisted Levi subgroups.
Then there is a bijection $\iota^\tn{rig}_\mathfrak{w}$
fitting into the commutative square
\[
\begin{tikzcd}[row sep=small,column sep=large]
\displaystyle\bigsqcup_{[N]} \frac{\Pi^\cE_\phi(N_\pm)}{W(G,N)(F)}
\rar{\iota^\tn{rig}_\mathfrak{w}}\dar &
\Phi^\cE_\phi(G) \dar \\
H^1_\tn{L-reg}(\cE_\tn{rig},G) \rar{\tau} &
\displaystyle\bigsqcup_{[N]}
\frac{X^{*}(\pi_{0}(Z(\widehat{G})^{\Gamma,+}_{(\widehat N)}))}{W(G,N)(F)},
\end{tikzcd}
\]
where $[N]$ ranges over stable conjugacy classes of
twisted Levi subgroups of~$G$
such that $\phi$ factors through ${}^LN_\pm\into{}^LG$.
\end{Theoremx}

\subsection{Notation} \label{sec:notation}
Throughout this article, $F$ is a nonarchimedean local field of characteristic zero.

\subsubsection*{Galois groups}
Let $\ov{F}$ be a fixed algebraic closure of~$F$
with absolute Galois group $\Gamma$ and Weil group $W_{F}$.
For a finite Galois subextension $E/F$ we denote by $\Gamma_{E/F}$ the Galois group of $E/F$.
The Weil Deligne group $W_{F} \times \SL_{2}$ of $F$ is denoted by $W_{F}'$.%

\subsubsection*{Cohomology}
For $G$ a finite type $F$-group scheme,
the notation $H^{1}(F,G)$ always refers to $\smash{H^{1}_{\text{fppf}}}(F, G)$.
In our context, $G$ will always be either reductive or commutative,
and is thus always smooth,
so we could just as well write $H^{1}_{\text{\'{e}t}}(F, G)$ or
$H^1(F,G)\defeq H^{1}(\Gamma, G(\ov{F}))$.
Given a group $\Omega$ acting on a possibly nonabelian group $G$,
the group $G$ acts on the left on $Z^1(\Omega,G)$:
define, for $g\in G$ and $x\in Z^1(\Omega,G)$,
the cocycle ${}^gx\colon \sigma \mapsto g\cdot x(\sigma)\cdot{}^\sigma g^{-1}$,
which we call the \textit{twist of $x$ by $g$}.
The quotient by this action is then $H^1(\Omega,G)$.
For $g\in G$, we write $dg\colon\sigma\mapsto g^{-1}\cdot{}^\sigma g$
for the coboundary of~$g$, a $1$-cocycle.

\subsubsection*{Algebraic groups}
For an affine algebraic group $G$ over $F$,
we denote by $Z(G)$ the center of $G$,
by $G^{\circ}$ its identity component,
and by $\pi_{0}(G)$ the quotient $G/G^{\circ}$.
For $G$ connected reductive,
we denote by $G_\tn{der}$ the derived subgroup of~$G$
and by $G_{\Sc}$ the simply-connected cover of~$G_\tn{der}$.
We denote conjugation by left superscript: ${}^gx\defeq gxg^{-1}$.
For $H \subseteq G$ a subgroup scheme,
we denote by $Z_{G}(H)$ and $N_{G}(H)$ the scheme-theoretic centralizer
and normalizer of $H$, respectively,
and for an $F$-rational homomorphism $f\colon G \to G'$,
we denote by $Z_{G}(f)$ the centralizer of the scheme-theoretic image of $f$.
We abbreviate $Z_G^\circ(H)\defeq Z_G(H)^\circ$.

Algebraic subgroups of $G$ are assumed to be $F$-subgroups,
so that, for instance, if $T$ is a maximal torus of~$G$
then $T$ is defined over~$F$.
If it becomes necessary to discuss groups over~$\ov{F}$,
then we will denote their base change to~$\ov{F}$ by a subscript,
such as $G_{\ov{F}}$.

A torus $T$ of~$G$ is \textit{elliptic}
(with respect to~$G$) if $T/Z(G)$ is anisotropic,
and a full-rank reductive subgroup $H$ of~$G$
is \textit{elliptic} if $Z(H)$ is elliptic.
For a maximal torus of $T$ of $G$,
we denote by $T_{\psc}$ the preimage of $T$ in $G_{\Sc}$,
or by $T_{\psc,G}$ to stress the dependence on $G$
or to allow another group to replace it.
The parentheses around ``sc'' are to avoid confusion with $T_\tn{sc}$,
which is trivial.

A reductive subgroup $M$ of~$G$ is \textit{full-rank}
if $M$ contains a maximal torus of~$G$.
Given a full-rank reductive subgroup $M$ of~$G$,
let $W(G,M)\defeq N_G(M)/M$ denote the relative Weyl group.
Two full-rank reductive subgroups $M$ and $M'$ of~$G$
are \textit{stably conjugate}
if there is $g\in G(\ov{F})$ such that $M' = {}^gM$ and $dg\in Z^1(F,M)$.

\subsubsection*{Dual groups}
We denote by $\smash{\widehat{G}}$ a complex dual group of connected reductive $G$,
identified with its $\C$-points.
Fix an $F$-pinning $(\mathscr{B}_{G}, \mathscr{T}_{G}, \{X_{\hat{\alpha}_{G}}\})$
of~$\widehat G$ and a $\Gamma$-stable pinning $(B_{G}, T_{G}, \{X_{\alpha_{G}}\})$ of $G$.
Let ${}^LG = \widehat G\rtimes\Gamma$ be the Galois form of the $L$-group.

\addtocontents{toc}{\protect\setcounter{tocdepth}{0}}
\subsection*{Acknowledgements}
The authors thank Alexander Bertoloni Meli for suggesting that the approach of \cite{BMOi23} be applied to the rigid framework and providing invaluable feedback throughout the writing of this paper. They also thank Tom Haines for many useful conversations about various technical points, especially related to the rigid Kottwitz map. The authors further thank Jeff Adams, Serin Hong, Tasho Kaletha, Alex Youcis, and Zhiwei Yun for helpful discussions. The GAP computer algebra system \cite{GAP4} helped direct our search for examples.

\addtocontents{toc}{\protect\setcounter{tocdepth}{2}}
\section{Regular $\mc{E}$-cohomology}\label{Ecohomsec}

\subsection{The rigid gerbe and band}\label{prelimsubsec}
Following \cite[Section~3.1]{Kaletha16a},
define the profinite group scheme of multiplicative type
\[
u \defeq \varprojlim_{E/F,n} u_{E/F,n},
\qquad u_{E/F,n} \defeq \Res_{E/F}(\mu_n)/\mu_n,
\]
where $n$ ranges over nonnegative integers,
$E$ ranges over finite Galois extensions of~$F$, and
the transition map for $E\subseteq K$ and $n\mid m$ is
the composition of the $K/E$-norm and multiplication by $m/n$.
That is, the transition map corresponds to the morphism of $\Gamma$-modules
\[
\frac{1}{m}\bbZ/\bbZ[\Gamma_{K/F}]_0 \to \frac{1}{n}\bbZ/\bbZ[\Gamma_{E/F}]_0\colon
\sum_\gamma c_\gamma [\gamma]
\mapsto \sum_\gamma \sum_{\sigma\mapsto\gamma} c_\gamma [\sigma],
\]
where the subscript $0$ denotes the kernel of the augmentation map.

\begin{Lemma}
The map $\varprojlim_{E/F,n} \Res_{E/F}(\mu_n) \to u$ is an isomorphism
and $X^*(u) \simeq C^\infty(\Gamma,\bbQ/\bbZ)$.
\end{Lemma}

\begin{proof}
For the first claim, apply the anti-equivalence $X^*$:
\[
X^*\bigl(\varprojlim_{E/F,n} \Res_{E/F}(\mu_n)\bigr)
\simeq \varinjlim_{E/F,n} \frac1n\bbZ/\bbZ[\Gamma_{E/F}]
\simeq \varinjlim_{E/F} \bbQ/\bbZ[\Gamma_{E/F}],
\]
and there is a similar expression for $u$ if we include the subscript~$0$.
It now suffices to show that the map
$\varinjlim_E \bbQ/\bbZ[\Gamma_{E/F}]_0
\to\varinjlim_E \bbQ/\bbZ[\Gamma_{E/F}]$ is surjective, hence an isomorphism.
This follows from the fact that for every finite Galois extension $E/F$ and every integer $n$,
there is a finite Galois extension $E'$ such that $n \mid [E':E]$.
That is, if $1/n[\gamma]\in\bbQ/\bbZ[\Gamma_{E/F}]$
then the image of this element in $\bbQ/\bbZ[\Gamma_{E'/F}]$ is
of the form $1/n[\sigma_1] + \cdots + 1/n[\sigma_{[E':E]}]$,
which lies in the kernel of the augmentation map.
For the second claim, identify $a[\sigma]\in\bbQ/\bbZ[\Gamma_{E/F}]$
with the indicator function taking value $a$ on the coset $\sigma\Gamma_E$
and $0$ otherwise.
\end{proof}

Let $\mu\defeq\varprojlim_n \mu_n$ with inverse limit taken over
$n$th-power maps, so that $u \simeq \varprojlim_{E/F} \Res_{E/F}(\mu)$.

\begin{Lemma} \label{bandhoms}
Let $A$ be a multiplicative type group.
\begin{enumerate}
\item
If $A$ is finite type, then
$\Hom_F(u,A) \simeq \varinjlim_n\Hom_{\ov F}(\mu_n,A)$.

\item
If $A$ is finite and $n$-torsion, then
$\Hom_F(u,A) \simeq \Hom_{\ov F}(\mu_n,A)$.

\item
If $A$ is finite, then $\Hom_F(u,A) \simeq \Hom_{\ov F}(\mu,A)$.
\end{enumerate}
\end{Lemma}

\begin{proof}
For the first part,
we first claim that 
\[
\textstyle\Hom_F\bigl(\varprojlim_{E/F,n}\Res_{E/F}(\mu_n),A\bigr)
= \varinjlim_{E/F,n}\Hom_F(\Res_{E/F}(\mu_n),A).
\]
After applying $X^*$, this claim follows from the fact that
finitely-generated $\Gamma$-modules are compact objects
in the category of $\Gamma$-modules.
Moreover, since restriction and induction
are bi-adjoint functors between $\Gamma_F$-modules
and $\Gamma_E$-modules,
$\Hom_F(\Res_{E/F}(\mu_n),A) \simeq \Hom_E(\mu_n,A)$.
For the second part, note that if $A$ is $n$-torsion
and $n$ divides $m$, then the map
$\Hom_E(\mu_m,A)\to\Hom_E(\mu_n,A)$
induced by multiplication by $m/n$ is an isomorphism.
The third part follows from the second part
and the definition of $\mu$ as an inverse limit.
\end{proof}

Using local class field theory,
one can show that $H^{1}(F,u) = 0$ and there is a canonical isomorphism
$H^{2}(F, u) \xrightarrow{\sim} \widehat{\Z}$ (see \cite[Theorem 3.1]{Kaletha16a}).

\begin{Definition}
For any cocycle $\xi \in Z^{2}(\Gamma, u(\ov{F}))$ with image $-1 \in H^{2}(F,u)$, we define the \textit{rigid gerbe} $\mc{E}$ to be the gerbe $\mc{E}_{\xi} \to \tn{Sch}/\tn{Spec}(F)$. 
\end{Definition}

For ease of calculations, since we are in mixed characteristic we may view $\mc{E}$ as a group extension
\begin{equation*}
0 \to u(\ov{F}) \to \mc{E} \to \Gamma \to 1.
\end{equation*}

The gerbe $\mc{E} \to  \tn{Sch}/\tn{Spec}(F)$ inherits the fpqc topology from the base, and it thus makes sense, for a finite type affine algebraic group $G$ over $F$, to consider fpqc $G_{\mc{E}}$-torsors on $\mc{E}$, which we will simply refer to as \textit{$G$-torsors on $\mc{E}$}. The set of such torsors is denoted by $Z^{1}(\mc{E}, G)$ and isomorphism classes by $H^{1}(\mc{E}, G)$. Under the dictionary between $\mc{E}$ and the associated group extension (see \cite[\S~7.1]{Taibi22}), which we will use freely, these torsors correspond to arbitrary $1$-cocycles $z$ of $\cE$ with coefficients in~$G(\ov F)$ such that the group homomorphism $z|_{u(\ov{F})}$ is algebraic. We will write $Z^1(\cE,G)$ for the set of such cocycles, and $H^1(\cE,G)$ for cohomology classes of such.

Although the choice of $\xi$ representing $-1 \in H^{2}(F, u)$ is far from canonical,
for any two choices $\xi, \xi'$ there is a non-canonical isomorphism of
$u$-gerbes $\mc{E}_{\xi} \to \mc{E}_{\xi'}$,
and the vanishing of $H^1(F,u)$ implies that the induced bijection
$H^{1}(\mc{E}_{\xi'}, G) \to H^{1}(\mc{E}_{\xi}, G)$
is independent of choices.
For this reason, it is harmless to fix a representative $\xi$
once and for all and set $\mc{E} = \mc{E}_{\xi}$. 

\subsection{Regular cohomology classes}\label{cohomsetdefs}
In this subsection we introduce a type of cohomology class of $\cE$,
the regular classes,
that generalizes basic cohomology classes and is our main focus.

Let $G$ be a connected reductive group over $F$.
Restricting $x\in Z^1(\cE,G)$ to $u$ yields a homomorphism
$x|_u\colon \smash{u_{\ov{F}} \to G_{\ov{F}}}$,
and cohomologous cocycles have conjugate restrictions:
$({}^gx)|_u = \Inn(g)\circ(x|_u)$.
Moreover, since $x$ is a $1$-cocycle,
the $G(\ov{F})$-conjugacy class of $x|_u$ is defined over~$F$.
Recall the following fundamental definition from \cite{Kaletha16a}:

\begin{Definition}
\begin{enumerate}
\item $x\in Z^1(\cE,G)$ is \textit{basic} if $x|_u$ factors through $Z(G)(\ov{F})$.
\item $[x]\in H^1(\cE,G)$ is \textit{basic} if one (equivalently, any)
representative of~$[x]$ is basic.
\end{enumerate}
\noindent Write $Z^1_\tn{bas}(\cE,G)$ and $H^1_\tn{bas}(\cE,G)$
for the sets of basic cocycles and cohomology classes.
\end{Definition}

In this article, we will generalize this definition in the following way:

\begin{Definition} 
\begin{enumerate}
\item $x\in Z^1(\cE,G)$ is \textit{regular} if
$x|_u$ is defined over~$F$ and
$x\in Z^1(\cE,Z_G^\circ(x|_u))$.
\item $[x]\in H^1(\cE,G)$ is \textit{regular}
if it lies in the image of $H^1(\cE,S)$
for some maximal torus $S$.
\end{enumerate}
\noindent Write $Z^1_\tn{reg}(\cE,G)$ and $H^1_\tn{reg}(\cE,G)$
for the sets of regular cocycles and cohomology classes.
\end{Definition}

Our applications to the local Langlands conjectures in Section~\ref{llcsec}
use a strengthening of regularity:

\begin{Definition}
\begin{enumerate}
\item
$x\in Z^1_\tn{reg}(\cE,G)$ is \textit{Levi regular} if
$Z_G^\circ(x|_u)$ is a Levi subgroup of~$G$.
\item
$[x]\in H_\tn{reg}^1(\cE,G)$ is \textit{Levi regular}
if it is cohomologous to a Levi-regular cocycle.
\end{enumerate}
\noindent Write $Z^1_\tn{L-reg}(\cE,G)$ and $H^1_\tn{L-reg}(\cE,G)$
for the sets of Levi-regular cocycles and cohomology classes.
\end{Definition}

The set $H^1_\tn{reg}(\cE,G)$ is functorial in~$G$
with respect to homomorphisms of~$F$-groups,
but the sets $H^1_\tn{bas}(\cE,G)$ and $H^1_\tn{L-reg}(\cE,G)$ are not.

\begin{Remark}
It is not true that every algebraic cocycle $x$ is regular,
even if we ask for $x|_u$ to be defined over~$F$.
In Lemma~\ref{nontoralcocycle}, for the group $G=\PGL_2$ with $p\neq2$
we construct an algebraic cocycle $x\in Z^1(\cE,G)$
for which $x|_u$ is defined over~$F$
and $x(u)$ is a form of the Klein four-group,
so that $Z_G(x|_u)$ is the symmetric group on four letters.
We exclude such cocycles because they are radically different from the regular ones.
However, it would be extremely interesting to find a role for these cocycles
in the local Langlands correspondence.
\end{Remark}

\begin{Proposition} \label{cocyclefac1}
\begin{enumerate}
\item Let $x\in Z^1(\cE,G)$.
If $x|_u$ is defined over~$F$, then $x\in Z^1(\cE,Z_G(x|_u))$.
\item $Z^1_\tn{bas}(\cE,G)\subseteq Z^1_\tn{reg}(\cE,G)$
and $H^1_\tn{bas}(\cE,G)\subseteq H^1_\tn{reg}(\cE,G)$.
\item If $x\in Z^1_\tn{reg}(\cE,G)$, then $x\in Z^1_\tn{bas}(\cE,Z_G^\circ(x|_u))$.
\item A class in $H^1(\cE,G)$ is regular if and only if
it is represented by a regular cocycle.
\item For $x\in Z^1_\tn{reg}(\cE,G)$,
the group $Z_G^\circ(x|_u)$ contains a maximal $F$-torus of~$G$,
is reductive, and is defined over $F$.
\item 
Let $x,y\in Z^1_\tn{reg}(\cE,G)$.
If $y = {}^gx$, then $dg\in Z^1(F,Z_G^\circ(x|_u))$.
Hence $Z_G^\circ(x|_u)$ and $Z_G^\circ(y|_u)$ are stably conjugate.
\end{enumerate}
\end{Proposition}

\begin{proof}
The first part follows from the crossed homomorphism property:
for $\dot\sigma\in\cE$ a lift of $\sigma\in\Gamma$
and $a\in u(\ov{F})$, we have that
$x({}^\sigma a) = x(\dot\sigma a\dot\sigma^{-1})
= x(\dot\sigma)\cdot{}^\sigma x(a)\cdot x(\dot\sigma)^{-1}$.
Since $x|_u$ is defined over~$F$,
the expression above equals ${}^\sigma x(a)$,
meaning that $x(\dot\sigma)$ must commute with $x(u)$.

For the second part, to show the inclusion on $Z^1$,
we need to show that if $x$ is basic,
then $x|_u$ is defined over~$F$.
This implication follows from the string of equations
in the proof of the first part:
there we have $x(\dot\sigma)\in Z(G)(\ov{F})$
and conclude that $x({}^\sigma a) = {}^\sigma x(a)$,
so that $x|_u$ is defined over~$F$.
The inclusion on $H^1$ in the first part will then follow from the fourth part.

The third part follows from the definitions:
the image of $x|_u$ lies in $Z_G^\circ(x|_u)$ and commutes with it,
hence lies in $Z(Z_G^\circ(x|_u))$.

Next, for the special case of the fifth part where $x\in Z^1(\cE,S)$,
with $S$ a maximal $F$-torus of~$G$,
we can see that the morphism $x|_u$, and thus $Z_G(x|_u)$
is defined over~$F$ by using the string of equations
from the proof of the first part, since $S$ is abelian.
The rest of the fifth part in this special case
follows from Springer's computation
of the centralizer of a subset of a reductive group
contained in a maximal torus \cite[2.14~Lemma]{Steinberg75}.

For the forward implication in the fourth part, if the class is regular,
then by definition we may choose a representative $x\in Z^1(\cE,S)$
for some maximal $F$-torus~$S$ of~$G$.
By the special case of the fifth part proved above,
$S$ is a maximal torus of $Z_G(x|_u)$,
hence is contained in $Z_G^\circ(x|_u)$,
so that $x\in Z^1(\cE,Z_G^\circ(x|_u))$ as well.
Conversely, if $x$ is a regular cocycle,
then by the third part it is a basic cocycle for $Z_G^\circ(x|_{u})$,
and thus, by \cite[Corollary~3.7(2)]{Kaletha16a},
it is cohomologous to an element of $H^1(\cE,S)$
for some elliptic maximal torus $S$ of $G$.

For the general case of the fifth part,
by the fourth part there is $x'\in Z^1(\cE,S)$ cohomologous to~$x$.
The groups $Z_G^\circ(x|_u)$ and $Z_G^\circ(x'|_u)$
are defined over~$F$ because $x|_u$ and $x'|_u$ are,
and the two groups are isomorphic over~$\ov{F}$
because any $g\in G(\ov{F})$
that takes $x$ to~$x'$ conjugates the first group to the second.

For the sixth part, $g$ conjugates $x|_u$ to $y|_u$
and thus $Z_G^\circ(x|_u)$ to $Z_G^\circ(y|_u)$.
For every $e\in\cE$,
\[
x(e) = g^{-1}\cdot y(e)\cdot {}^eg = g^{-1} x(e) g\cdot dg(e),
\]
and since $x(e),g^{-1}y(e)g\in Z_G^\circ(x|_u)(\ov{F})$
because $x$ and $y$ are regular,
also $dg(e)\in Z_G^\circ(x|_u)(\ov{F})$.
\end{proof}

It will also be useful to define a ``finite-level''
version of $H^1_\tn{reg}(\cE,G)$,
where the image of the band is constrained.
Given a finite central subgroup $Z$ of~$G$,
we write $H^1(\cE,Z\to G)$ for the cohomology classes
for which some (equivalently, any) representative cocycle~$x$
has the property that $x|_u$ factors through~$Z$.

\begin{Definition}
Let $A$ be a multiplicative-type subgroup of~$G$ contained in a maximal torus.
We define $H^1_\tn{reg}(\cE,A\to G)$
as the elements lying in the image of $H^1(\cE,A\to S)$
for some maximal torus~$S$ of~$G$ containing~$A$.
\end{Definition}

It is proved in \cite[Lemma 3.6, Remark 3.11]{BMD26} that, in the special case $G = \mathrm{GL}_{n}$, the set $H^{1}(\cE, G)$ is exceptionally well-behaved:

\begin{Proposition}
Every class in $H^{1}(\cE, \mathrm{GL}_{n})$ is regular, and every class has a regular representative $x$ with $Z_{\mathrm{GL}_{n}}(x|_{u}) = \mathrm{Res}_{E/F}(\mathrm{GL}_{s})$ for some $E/F$ a finite extension such that $[E \colon F] \cdot s = n$, viewed as a subgroup of $\mathrm{GL}_{n}$ in the usual way.
\end{Proposition}

\subsection{The Newton decomposition}
The sixth part of Proposition~\ref{cocyclefac1}
associates to each regular cohomology class
a stable conjugacy class of full-rank reductive subgroups of~$G$,
which are twisted Levi subgroups if, in addition,
the cohomology class is Levi regular.
A regular cocycle lies in $H^1_\tn{reg}(\cE,A\to G)$
if and only if it is represented by some $x\in Z^1_\tn{reg}(\cE,G)$
for which $x|_u$ factors through~$A$.

\begin{Definition}
A full-rank reductive subgroup $M$ of~$G$
is a \textit{rigid Newton centralizer}
if $M = Z_G^\circ(x|_u)$ for some $x\in Z^1_\tn{reg}(\cE,G)$.
\end{Definition}

\noindent We will study but not completely classify
rigid Newton centralizers in Section~\ref{admsbgpsubsec}.

The construction of rigid Newton centralizers
from regular cohomology classes provides a convenient way to organize $H^1_\tn{reg}(\cE,G)$,
which plays a key role in several of our constructions.

\begin{Definition}
Let $M$ be a full-rank reductive subgroup of~$G$.
\begin{enumerate}
\item Let $H^1_\tn{reg}(\mc{E},G)_M$
be the classes with regular representative $x$
such that $Z_G^\circ(x|_u) = M$,
or equivalently, such that $Z_G^\circ(y|_u)$
is stably conjugate to $M$ for any regular representative~$y$.

\item Let $H^1_\tn{bas}(\mc{E},M)_\Greg$
be the classes such that $Z_G^\circ(x|_u)=M$
for some (equivalently, any) representative~$x$.
We call such classes \textit{$G$-regular}.
\end{enumerate}
\end{Definition}

So $H^1_\tn{reg}(\cE,G)_M$ is nonempty precisely when $M$
is a rigid Newton centralizer.
This construction yields a partition of $H^1_\tn{reg}(\mc{E},G)$,
which we call the \textit{Newton decomposition}:
\[
H^1_\tn{reg}(\mc{E},G) = \bigsqcup_{[M]} H^1_\tn{reg}(\mc{E},G)_M,
\]
the disjoint union ranging over stable conjugacy classes~$[M]$
of rigid Newton centralizers $M$ of~$G$.
In fact, elements in the component for~$M$
can be reduced to basic classes for~$M$,
as we now explain.

The short exact sequence defining the relative Weyl group
$W(G,M)$ gives rise to an action $W(G,M)(F)\actson H^1_\tn{reg}(\mc{E},M)$
by the following classical construction \cite[I\S~5.5]{Serre02}:
for $w \in W(G,M)(F)$ and $x\in Z^1(\mc{E},M)$,
choose a lift $\dot w \in N_{G}(M)(\ov{F})$ of~$w$
and set
\begin{equation}\label{Weyltransaction}
w\cdot[x] = [{}^{\dot w}x].
\end{equation}
Since $N_G(M)$ normalizes $Z(M)$,
this action also stabilizes the subset $H^1_\tn{bas}(\mc{E},M)$.

\begin{Theorem} \label{newtondecomp}
For every rigid Newton centralizer~$M$,
the fibers of the map $H^1_\tn{bas}(\cE,M)_\Greg\to H^1_\tn{reg}(\cE,G)$
are orbits for the action of $W(G,M)(F)$.
Hence there is a decomposition
\[
H^1_\tn{reg}(\cE,G) \simeq \bigsqcup_{[M]}
\frac{H^1_\tn{bas}(\cE,M)}{W(G,M)(F)}
\]
where $[M]$ ranges over stable conjugacy classes of rigid Newton centralizers.
\end{Theorem}

\begin{proof}
By Lemma~\ref{cocyclefac1}, the natural map
$H^1_\tn{reg}(\cE,M)\to H^1_\tn{reg}(\cE,G)$ restricts to a surjection
\[
\begin{tikzcd}
H^1_\tn{bas}(\mc{E},M)_\Greg \rar[twoheadrightarrow] &
H^1_\tn{reg}(\mc{E},G)_M.
\end{tikzcd}
\]
We will show that this map is also injective
after dividing out the target by~$W(G,M)(F)$.

Let $x,y\in Z^1_\tn{bas}(\cE,M)_\Greg$
and suppose that the images of $x$ and $y$ are cohomologous in~$G$,
so that $y = {}^gx$ for $g\in G(\ov{F})$.
Then $g$ normalizes~$M$ because it conjugates
$Z_G^\circ(x|_u)$ to $Z_G^\circ(y|_u)$ and these both equal~$M$.
Moreover, the image of $g$ in $W(G,M)(\ov{F})$
is defined over~$F$ because $dg\in Z^1(F,M)$
by the sixth part of Proposition~\ref{cocyclefac1}.
Hence, by \eqref{Weyltransaction},
the classes $[x]$ and $[y]$
lie in the same $W(G,M)(F)$-orbit.
Hence $H^1_\tn{bas}(\cE,M)_\Greg/W(G,M)(F)\to H^1_\tn{reg}(\cE,G)_M$
is injective.
\end{proof}

There is a further refinement of the Newton decomposition
analogous to the construction of extended pure inner twists
in the theory of the Kottwitz set.

\begin{Notation}
Given $x\in Z^1_\tn{reg}(\cE,G)$,
let $G_x$ be the inner form of $Z_G^\circ(x|_u)$
obtained from twisting this group by the image of $x$
in $H^1(F,(Z_G^\circ(x|_u))_\tn{ad})$.
The inner isomorphism equivalence class of this group depends only on~$[x]$,
and we denote it by~$G_{[x]}$.
\end{Notation}

For later use, we record:

\begin{Lemma}\label{miracleLem}
Let $M$ be a quasi-split twisted Levi subgroup of~$G$.
If $[y] = w \cdot [x] \in H^{1}_{\tn{bas}}(\mc{E}, M)$ for $w \in W(G,M)(F)$,
then $[x]$ and $[y]$ have the same image in $H^{1}(F, M_{\tn{ad}})$.
\end{Lemma}	

\begin{proof}
By twisting we can reduce to the case when $[x]$ is the neutral class and $[y]$ has representative $d\dot w$ for $\dot w$ lifting $w$.
A choice of pinning of $M$ gives rise to a splitting of the short exact sequence
\begin{equation*}
1 \to \underline{\tn{Inn}}_{M} \to \underline{\tn{Aut}}_{M}
\to \underline{\tn{Out}}_{M} \to 1.
\end{equation*}
Hence any $w \in W(G,M)(F) \subseteq \tn{Out}_{M}(F)$
can be lifted to an $F$-automorphism $\theta_{w} \in \tn{Aut}_{M}(F)$.
The morphism $\Inn(\dot w) \in \underline{\tn{Aut}}_{M}(\ov{F})$
has the same image in $\underline{\tn{Out}}_{M}(F)$ as $\theta_{w}$
and so there is some $m \in M(\ov{F})$ such that $\Inn(\dot wm)$ is defined over $F$.
Now we can use $d(\dot wm)$ as the representative for $[y]$,
and we just saw that it has trivial image in $H^{1}(F, M_{\tn{ad}})$.
\end{proof}

\section{The Newton map}
Simply stated, the Newton map restricts an element of
$H^1_\tn{reg}(\cE,G)$ to the band~$u$,
yielding a conjugacy class of homomorphisms $u\to G$.
The difficulty in studying this map is to make precise its target,
which is much more complex than in the isocrystal case.

\subsection{Defining the map}
\begin{Definition}
\begin{enumerate}
\item[]
\item
Let $\Hom_F(u,G)/_\tn{st}G$ be the set of homomorphisms $u\to G$
modulo the following equivalence relation:
$\lambda\sim\lambda'$ if there is $g\in G(\ov{F})$ such that
$\lambda' = {}^g\lambda$ and $dg\in Z^1(F,Z_G^\circ(\lambda))$.

\item
Given a maximal $F$-torus $T\subseteq G$ with Weyl group $W$,
define $\cC_\tn{st}(G,T)$ as the quotient of $X_*(T)\otimes\bbQ/\bbZ$
by the equivalence relation
$\lambda\sim\lambda'$ if there is $w\in W$ such that
\[
\tn{$\lambda'=w\cdot\lambda$ and
$dw\in Z^1(F,Z^\circ_W(\Gamma\cdot\lambda))$,}
\]
where, for $A\subseteq X_*(T)$,
we let $Z^\circ_W(A)$ be the subgroup generated by
reflections $s_\alpha$ for $\alpha\in R(G,T)$ such that
$\langle\alpha,a\rangle = 0$ for all $a\in A$.
\item
Define $\cC_\tn{st}(G)$ as the quotient of
$\bigsqcup_{T\subseteq G} X_*(T)\otimes\bbQ/\bbZ$,
the union taken over all maximal $F$-tori of~$G$,
by the equivalence relation
$(\lambda,T)\sim(\lambda',T')$ if there is $g\in G(\ov F)$ such that
\[
\tn{$(\lambda',T') = g(\lambda,T)g^{-1}$ and
$dg\in Z^1(F,N_G(T)\cap Z_G^\circ(\Gamma\cdot\lambda))$.
}
\]
\end{enumerate}
\end{Definition}

The letter $\cC$ stands for ``chamber'',
since in many (but not all cases) a subset of $\cC_\tn{st}(G,T)$
can be identified with an affine Weyl chamber,
and the letters ``st'' stand for ``stable'',
since the equivalence relation defining $\cC_\tn{st}(G,T)$
is stable conjugacy.

\begin{Lemma}
For every $T\subseteq G$, the map
$\cC_\tn{st}(G,T)\to\cC_\tn{st}(G)$ is injective.
As $T$ ranges over representatives of stable conjugacy classes of maximal tori,
these maps jointly surject onto $\cC_\tn{st}(G)$.
\end{Lemma}

\begin{proof}
Joint surjectivity is clear.
For injectivity, note first that
$(N_G(T)\cap Z_G^\circ(\Gamma\cdot\lambda))(\ov{F})
\simeq Z_W^\circ(\Gamma\cdot\lambda)$
by Steinberg's description of centralizers
\cite[2.14~Lemma]{Steinberg75}.
If $(\lambda,T)$ and $(\lambda',T)$
are equivalent in $\cC_\tn{st}(G)$,
then there is $g\in N_G(T)(\ov F)$ such that
$dg\in Z^1(F,N_G(T)\cap Z_G^\circ(\Gamma\cdot\lambda))$,
and the image of~$g$ in $W(G,T)(\ov F)$
witnesses the equivalence of $\lambda$ and~$\lambda'$
in $\cC_\tn{st}(G,T)$.
\end{proof}

\begin{Lemma} \label{newtonmaptarget}
Restriction to the identity component yields an isomorphism
\[
\Hom_F(u,G) /_\tn{st} G \simeq \cC_\tn{st}(G).
\]
\end{Lemma}

\begin{proof}
Let $E/F$ be a finite Galois extension.
We can identify $\Hom_{\ov F}(\Res_{E/F}(\mu_n),G)$
with the set of functions $f\colon\Gal(E/F)\to \Hom_{\ov F}(\mu_n,G)$.
The group $\Gamma$ acts on these functions
through its projection to~$\Gal(E/F)$,
by $(\sigma f)(\tau) = \sigma f(\sigma^{-1}\tau)$.
Hence $f$ is defined over~$F$ if and only if
$\sigma f(1) = f(\sigma)$ for all $\sigma\in\Gal(E/F)$.
In this case, we can recover $f$ from $f(1)\in G(E)$,
reflecting the adjunction
$\Hom_F(\Res_{E/F}(\mu_n),G) = \Hom_E(\mu_n,G)$.
Note, however, that if $f$ is not defined over~$F$
then we cannot recover $f$ from $f(1)$.

We define the map as follows.
Given a function $f$ as above which is defined over~$F$,
let $T$ be a maximal $F$-torus containing~$f(1)$.
Send the class of $f$ to the class of $(f(1),T)$.

Let us check that the map is well defined.
Note that $Z_G(f) = Z_G(\Gamma\cdot f(1))$,
as $f$ is defined over~$F$.
First, suppose $T'$ is another torus through which $f(1)$ factors.
Since $T$ and $T'$ are maximal tori of~$Z_G^\circ(f)$,
we can conjugate $T$ to $T'$ by some $g\in Z_G^\circ(f)(\ov F)$,
which satisfies $dg \in Z^1(F,N_G(T)\cap Z_G^\circ(f))$
because $T'$ is defined over~$F$.
Hence $(f(1),T)$ and $(f(1),T')$
represent the same element of $\cC_\tn{st}(G)$.
Second, suppose $f'$ is defined over~$F$ and
$f' = {}^g f$ for some $g\in G(\ov F)$ such that $dg\in Z^1(F,Z_G^\circ(f))$.
This condition still holds if we multiply $g$
on the left by an element of $Z_G^\circ(f')(\ov F)$,
meaning that we may assume ${}^g T$ is defined over~$F$.
Then $dg\in Z^1(F,Z_G^\circ(f)\cap N_G(T))$, meaning that
$(f(1),T)$ is equivalent to $(f'(1),T')$ in $\cC_\tn{st}(G)$.
\end{proof}

\begin{Definition}
The \textit{Newton map} $\nu\colon H^1_\tn{reg}(\cE,G)\to\cC_\tn{st}(G)$
is the composition of the restriction map
$H^1_\tn{reg}(\cE,G)\to\Hom_F(u,G)/_\tn{st} G$
with the isomorphism of Lemma~\ref{newtonmaptarget}.
\end{Definition}

\noindent Evidently $\cC_\tn{st}(G)$ is functorial in~$G$
and $\nu$ is natural in~$G$.

As the terminology suggests, the rigid Newton centralizer
for a regular cohomology class can be read off from its image under the Newton map.
If $x$ is a regular representative and 
$T$ is a maximal torus of $Z_G^\circ(x|_u)$,
and $\lambda\in\cC_\tn{st}(G,T)$ represents $\nu([x])$,
then $Z_G^\circ(x|_u)$ is stably conjugate to
\[
Z_G^\circ(\Gamma\cdot\lambda)
= \bigcap_{\sigma\in\Gamma} Z_G^\circ({}^\sigma\lambda).
\]
The equivalence relation on $\cC_\tn{st}(G)$
is designed so that any two choices of representative $\nu([x])$
yield stably conjugate subgroups $Z_G^\circ(\Gamma\cdot\lambda)$.
We can use this construction to define a Newton decomposition for $\cC_\tn{st}(G)$:
Given a rigid Newton centralizer~$M$ with maximal torus~$T$,
let $\cC_\tn{st}(G)_M$ be the subset of points of $\cC_\tn{st}(G)_M$
with some representative $\lambda$ for which $Z_G^\circ(\Gamma\cdot\lambda)=M$.
Then
\[
\cC_\tn{st}(G) = \bigsqcup_{[M]} \cC_\tn{st}(G)_M,
\]
the disjoint union ranging over stable conjugacy classes~$[M]$
of rigid Newton centralizers $M$ of~$G$.
Moreover, $\nu$ restricts to a map
$H^1_\tn{reg}(\cE,G)_M\to\cC_\tn{st}(G)_M$.

It is difficult to give a simple combinatorial description of $\cC_\tn{st}(G)_M$,
but one can get some feel for the global structure of $\cC_\tn{st}(G)$
by considering the extreme values of~$M$.

\begin{Remark}
At one extreme, if $\lambda$ is regular,
then $Z_G(\Gamma\cdot\lambda) = T$,
and the equivalence class of $\lambda$ in $\cC_\tn{st}(G,T)$
can be identified with the rational Weyl orbit $W(F)\cdot\lambda$.
Hence the ``dense subset'' $\cC_\tn{st}(G,T)_\tn{reg}\subseteq\cC_\tn{st}(G,T)$
of regular elements can be identified with
\[
\bigl(X_*(T)\otimes\bbQ/\bbZ\bigr)_\tn{reg}/W(F),
\]
though this subset is typically strictly smaller than $\cC_\tn{st}(G,T)_\tn{reg}$.
When $T$ is split, or anisotropic with $\Gamma$
acting by negation through a quadratic character, say,
we can interpret this quotient as the rational points
in an affine Weyl chamber.
In general, there are examples of $T\subseteq G$
for which $W(F)$ is trivial, for instance,
$E^\times \subseteq \GL_n(F)$ if $E/F$ is a degree-$n$ separable extension
with $\Aut(E/F)=1$,
and in these cases, this subset is a union of such chambers.
Finally, $\cC_\tn{st}(G,T)_\tn{reg}$ is identified with
$\cC_\tn{st}(G,T')_\tn{reg}$ in $\cC_\tn{st}(G)$
if and only if $T$ and $T'$ are stably conjugate.
At the other extreme, if $\lambda=1$ is the trivial torsion cocharacter,
then all elements of the form $(1,T)$ are identified in $\cC_\tn{st}(G)$.
Hence the cells $\cC_\tn{st}(G,T)$ are glued together at least at~$1$,
and often at other points besides.
\end{Remark}

\subsection{Classifying rigid Newton centralizers}\label{admsbgpsubsec}
A natural question is which subgroups of $G$ arise as $Z_G^\circ(x|_u)$
for $x \in Z^1_{\tn{reg}}(\cE, G)$, the rigid Newton centralizers.
By design, every regular class is in the image of $H^{1}(\cE, T)$
for some $F$-rational maximal torus $T$ of $G$.
We can thus compute all possible rigid Newton centralizers in $G$
by ranging over all $T$ and studying the Newton map of $H^{1}(\cE, T)$.

Recall from \cite[Theorem 4.8]{Kaletha16a} (or Theorem~\ref{Tatenak1tori})
the Tate--Nakayama duality isomorphism $\iota_{[Z\to T]}$
for $H^1(\cE,Z\to T)$.

\begin{Lemma} \label{Lemtncompat}
Let $T$ be a torus,
let $Z\subseteq T$ be a finite multiplicative subgroup,
let $E/F$ be a finite Galois extension splitting $T$
and such that $N_{E/F}(X^*(Z))=0$,
and let $n$ be a multiple of the order of $X^*(Z)$.
The Tate--Nakayama duality isomorphism $\iota_{[Z \to T]}$
fits into the following diagram,
which is commutative with exact bottom row:
\begin{equation}\label{TNbanddiag}
\begin{tikzcd}
H^{1}(\cE, Z \to T) \dar{\sim}[swap]{\iota_{[Z\to T]}} \arrow{r} &
\Hom_{\ov{F}}(\mu_n, Z) \dar{\sim} &  \\
\dfrac{X_{*}(T/Z)^{N_{E/F}}}{I_{E/F}\cdot X_{*}(T)} \arrow{r} &
\dfrac{X_*(T/Z)}{X_*(T)} \arrow{r} &
\dfrac{X_{*}(T)^{\Gamma}}{N_{E/F}(X_{*}(T))},
\end{tikzcd}
\end{equation}
\end{Lemma}

\begin{proof}
The bottom row is \cite[Fact~4.1]{Kaletha16a},
noting, by our assumption on~$E/F$,
that every element in the middle group lies in the kernel of $N_{E/F}$.
The right vertical arrow is the inverse of \cite[(4.1)]{Kaletha16a},
and the square commutes by the final part of \cite[Theorem~4.8]{Kaletha16a}.
\end{proof}

\begin{Proposition}\label{Newtonimageprop}
Let $T$ be a torus, let $n\geq1$,
and let $E/F$ be a finite Galois extension splitting~$T$.
A map $\nu \in \Hom_{\ov{F}}(\mu_n, T[n])$
lies in the image of the restriction map from $H^{1}(\cE, T[n] \to T)$
if and only if $\nu$ extends to an element of $X_{*}(T)^{N_{E/F}}$.
\end{Proposition}

\begin{proof}
First, take $E/F$ such that $N_{E/F}(X^*(Z)\otimes\bbZ/n\bbZ)=0$,
so that we may apply Lemma~\ref{Lemtncompat}.
The isomorphism $T/T[n] \to T$ induced by the $n$th power map
gives a more concrete description
of the bottom row of \eqref{TNbanddiag} for $Z=T[n]$.
Under this identification,
\[
\frac{X_{*}(T/T[n])^{N_{E/F}}}{I_{E/F}X_{*}(T)}
= \frac{X_{*}(T)^{N_{E/F}}}{I_{E/F}[nX_{*}(T)]}, \qquad
\frac{X_{*}(T/T[n])}{X_{*}(T)} = \frac{X_{*}(T)}{nX_{*}(T)},
\]
the left bottom-row map sends the class of $\lambda \in X_{*}(T)$
to the image of $\lambda|_{\mu_n}$,
and if the left vertical map takes $\nu$ to the class of
$\lambda\in X_*(T)$, then the right bottom-row map
sends the class of $\lambda$ to the class of $N_{E/F}(\lambda)$.
It follows that the kernel of the right bottom-row map is identified with the
$\nu \in \Hom_{\ov{F}}(\mu_n, T[n])$
which have a lift $\lambda \in X_{*}(T)$ such that
$N_{E/F}(\lambda) = N_{E/F}(n\lambda')$ for some $\lambda' \in X_{*}(T)$.
Then the map $\lambda - n \lambda'$ lies in $X_{*}(T)^{N_{E/F}}$
and has the same restriction to $\mu_n$ as $\lambda$,
since $\lambda'|_{\mu_n}$ takes values in $T[n]$. 
To finish, where $E/F$ is the splitting field of~$T$,
note that $X_*(T)^{N_{E'/F}}$
is the same for all finite extensions $E'/F$ splitting~$T$.
\end{proof}

We can use Proposition \ref{Newtonimageprop} to show:

\begin{Theorem}\label{twistedLevithm}
Every elliptic twisted Levi subgroup of $G$ is a rigid Newton centralizer.
\end{Theorem}

\begin{proof}
Let $M$ be the twisted Levi subgroup and let $T = Z^\circ(M_\tn{der})$.
The torus $T$ is anisotropic because $M$ is elliptic in~$G$,
and thus $N_{E/F}(X_*(T)) = 0$ for every finite Galois extension $E/F$
splitting~$T$.

We claim that there is a homomorphism $\nu\colon\mu_n\to T_{\ov{F}}$
such that $Z_G(\nu) = M$.
This follows from Steinberg's description of
the centralizer of a semisimple element
\cite[2.14~Lemma]{Steinberg75}.
Indeed, choose a maximal torus $S$ of~$G$ containing~$T$.
Since $Z_G(T) = M$,
as $\alpha$ ranges over $R(G,S)\setminus R(M,S)$
the proper subtori $\ker(\alpha)\cap T_{\ov{F}}\subseteq T_{\ov{F}}$
do not cover $T_{\ov{F}}$.
So as soon as $\nu$ avoids these tori
we will have $Z_G^\circ(\nu) = M$.
Moreover, as soon as the order of~$\nu$
on every irreducible factor of~$G_\tn{der}$
is not bad for~$G$, its centralizer will be connected,
proving the claim.

Now let $\nu$ be as in the claim.
By Proposition~\ref{Newtonimageprop},
the map $\nu$ extends to a cocycle $x\in Z^1(\cE,T)$.
Then $Z_G(x|_u)$ is the intersection of all Galois conjugates of~$M$,
but since $M$ is defined over~$F$, this is $M$ itself.
\end{proof}

The converse of Theorem~\ref{twistedLevithm} is false:
there are rigid Newton centralizers that are non-elliptic twisted Levi subgroups
(Example~\ref{exnotell}) or that are not twisted Levi subgroups at all
(Example~\ref{nonadmex}).

\section{The Kottwitz map}\label{Kottmapsec}
Classically, the Kottwitz map is an isomorphism
$H^1(F,G)\longrightarrow \pi_1(G)_{\Gamma,\tn{tor}}$
where $\pi_1(G)$ is Borovoi's algebraic fundamental group,
the finitely-generated $\Gamma$-module defined by
\[
\pi_1(G) \defeq \varinjlim_{S\subseteq G} \frac{X_*(S)}{X_*(S_\psc)},
\]
in which the colimit ranges over maximal $F$-tori of~$G$
and the transition maps from $S$ to~$S'$, all isomorphisms,
are induced by the elements of $G(\ov{F})$ conjugating $S$ to~$S'$
(see \cite[11.3.1 and 11.7.7]{KalethaPrasad}).
Similarly, for the isocrystal gerbe one has a Kottwitz map (see \cite[\S~11]{Kottwitz14})
\begin{equation*}
B(G) \to \pi_{1}(G)_{\Gamma},
\end{equation*}
which now is no longer necessarily an isomorphism.
In both settings, the target of the Kottwitz map admits a dual description:
\[
\pi_1(G)_{\Gamma,\tn{tor}} \simeq X^*(\pi_0(Z(\widehat G))^\Gamma),\qquad
\pi_1(G)_\Gamma \simeq X^*(Z(\widehat G)^\Gamma).
\]

The goal of this section is to generalize these constructions to the rigid setting.
More precisely, in Section~\ref{kotttargsec}
we construct a set $\ov{Y}_{+,\tn{tor}}(G)$
generalizing Kaletha's $\ov{Y}_{+,\tn{tor}}(Z\to G)$,
and in Section~\ref{kottdefsec} we define a map
\begin{equation}\label{Kottinteq1}
H^{1}_{\tn{reg}}(\cE, G) \to \ov{Y}_{+,\tn{tor}}(G).
\end{equation}
To motivate our constructions, which are somewhat elaborate,
we first recall the basic case of \eqref{Kottinteq1},
due to Kaletha \cite{Kaletha16a}.

Define the category $\mathcal{R}$ whose objects are pairs $[A \to G]$
with $G$ a reductive $F$-group, $A$ a finite multiplicative $F$-group,
and $A\to G$ a monomorphism that factors through some $F$-rational torus.
The morphisms $[A_{1} \to G_{1}] \to [A_{2} \to G_{2}]$
in $\mathcal{R}$ are $F$-rational homomorphisms
$h\colon G_{1} \to G_{2}$ sending $A_{1}$ into $A_{2}$
and such that $Z_{G_{2}}^\circ(h(A_{1})) = Z_{G_{2}}^\circ(A_{2})$.
Kaletha uses the full subcategory $\mathcal{R}_\tn{cent}$ of $\mathcal{R}$
consisting of the objects $[Z \to G]$ with $f(Z)$ central.
The assignment $[Z \to G] \mapsto H^{1}(\cE, Z \to G)$
is a functor from $\mathcal{R}_\tn{cent}$ to sets,
or even, as follows from the theory of the basic Kottwitz map
(see Theorem \ref{Tatenak1} below), to abelian groups.

Kaletha first constructs the functor $\ov{Y}_{+,\tn{tor}}$
on the full subcategory $\mathcal{T}$ of $\mathcal{R}_\tn{cent}$ where $G$ is a torus:
\begin{equation}\label{linalgfuncbas}
\ov{Y}_{+,\tn{tor}}(Z \to S)
\defeq \varinjlim_{E/F} \frac{X_{*}(S/Z)^{N}}{I\cdot X_{*}(S)}
\end{equation}
where $I = I_{E/F}$ is the augmentation ideal for $\Gamma_{E/F}$,
the superscript $N$ denotes the elements killed by the norm map $N_{E/F}=N$,
and the colimit is taken over finite separable extensions $E/F$ splitting~$S$.
Equivalently, $\ov{Y}_{+,\tn{tor}}(Z \to S)$
is the torsion subgroup of $X_{*}(S/Z)^N/I\cdot X_{*}(S)$, for any such $E/F$.

\begin{Theorem}[{\cite[Theorem 4.8]{Kaletha16a}}]\label{Tatenak1tori}
There is a canonical isomorphism of functors on $\mathcal{T}$:
\begin{equation*}
H^{1}(\cE, -) \isoto \ov{Y}_{+,\tn{tor}}(-).
\end{equation*}
\end{Theorem} 
The functor $\ov{Y}_{+,\tn{tor}}$ extends to $\mathcal{R}_\tn{cent}$ via the formula below,
which is formally similar to the definition of~$\pi_1(G)_\Gamma$:
\begin{equation} \label{oldybar}
\ov{Y}_{+,\tn{tor}}(Z \to G) \defeq \varinjlim_{S \subseteq G} \varinjlim_{E/F} \frac{(X_{*}(S/Z)/X_{*}(S_{\psc}))^{N}}{I(X_{*}(S)/X_{*}(S_{\psc}))},
\end{equation}
Here the inner limit is as in \eqref{linalgfuncbas},
the outer limit ranges over all maximal $F$-rational tori $S$ of $G$,
and the outer transition maps from $S$ to $S'$
are isomorphisms induced by the elements of~$G(\ov{F})$
conjugating $S$ to~$S'$ (see \cite[Lemma~4.2]{Kaletha16a}).
This construction gives a well-defined functor
$\ov{Y}_{+,\tn{tor}}$ from $\mathcal{R}_\tn{cent}$ to abelian groups,
as shown in \cite[\S~4.1]{Kaletha16a},
and this functor computes basic cohomology of the rigid gerbe:

\begin{Theorem}[{\cite[Theorem 4.11]{Kaletha16a}}]\label{Tatenak1}
The map of Theorem~\ref{Tatenak1tori} extends to an isomorphism
of functors on $\mathcal{R}_\tn{cent}$:
\begin{equation*}
\iota_{-} \colon H^{1}(\cE, -) \longisoto \ov{Y}_{+,\tn{tor}}(-)
\end{equation*}
\end{Theorem} 

\begin{Remark}
In the theory of the Kottwitz map for $B(G)$,
both the basic and non-basic versions of the map have the same target,
$\pi_1(G)_\Gamma$, and the map $B(G)\to\pi_1(G)_\Gamma$ is far from a bijection.
In our rigid setting, however, the map
$H^1_\tn{reg}(\cE,G)\to\ov{Y}_{+,\tn{tor}}(G)$
that we call the Kottwitz map
has a different target than the basic Kottwitz map
$H^1_\tn{bas}(\cE,G)\isoto\ov{Y}_{+,\tn{tor}}(Z(G)\to G)$,
and our ``Kottwitz map'' will turn out to be injective
(Proposition~\ref{kottinj}).
The fundamental source of this discrepancy
is the failure of $\ov{Y}_{+,\tn{tor}}(Z(G)\to G)$
to be functorial in~$G$ (see Section~\ref{kottfuncex}),
and it is unclear to us how to develop
the rigid theory to formally match the isocrystal theory.
\end{Remark}

\subsection{The target of the map}\label{kotttargsec}
Our first step is to extend $\ov{Y}_{+,\tn{tor}}$
from $\mathcal{R}_\tn{cent}$ to all of $\mathcal{R}$.
For ease of notation, given $[A \to G] \in \text{Ob}(\mathcal{R})$,
we identify $A$ with $f(A) \subseteq G$ to treat $A$
as an $F$-rational multiplicative subgroup of $G$.
The basic idea is to replace $Z$ with $A$ in \eqref{oldybar},
taking care to define the direct limit correctly.

\begin{Lemma} \label{trivaction}
Let $S\subseteq G$ be a maximal torus and $A\subseteq S$ a finite subgroup.
\begin{enumerate}
\item $N_G(S)(\ov{F})$ acts trivially on $X_*(S)/X_*(S_\psc)$.
\item $(Z_G^\circ(A)\cap N_G(S))(\ov{F})$ acts trivially on $X_*(S/A)/X_*(S_\psc)$.
\end{enumerate}
\end{Lemma}

\noindent In the second part we cannot replace $Z_G^\circ(A)$ by $Z_G(A)$:
when $G=\PGL_2$ and $A=\mu_2$,
the group $Z_G(A) = N_G(S)$ acts nontrivially, by negation,
on $X_*(S/A)/X_*(S_\psc)\simeq\bbZ/4\bbZ$.

\begin{proof}
The first part is well known, and implicit in the definition of $\pi_1(G)$;
see, for example, the proof of \cite[Lemma~4.2]{Kaletha16a}.
For the second part, let $H=Z_G^\circ(A)$.
Applying the first part to the maximal torus $S/A$
of the reductive group $H/A$,
we see that $N_{H/A}(S/A)(\ov{F})$ acts trivially on
$X_*(S/A)/X_*(T)$, where $T$ is the inverse image of $S/A$ in
$(H/A)_\tn{sc} = H_\tn{sc}$.
The inclusion $H\into G$ induces a map $H_\tn{sc}\to G_\tn{sc}$,
whence a map $T\to S_\psc$ and an inclusion $X_*(T)\into X_*(S_\psc)$.
So $X_*(S/A)/X_*(S_\psc)$ is a quotient of $X_*(S/A)/X_*(T)$,
and thus $N_H(A)(\ov{F})$ acts trivially on it.
\end{proof}

Now we turn to defining the inverse limit.

\begin{Lemma}\label{stabisom}
Let $g\in G(\ov{F})$, let $S\subseteq G$ be a maximal torus,
and let $A\subseteq S$ be a finite subgroup.
If $dg\in Z^1(F,Z_G^\circ(A)\cap N_G(S))$,
then the isomorphism
\[
\frac{X_*(S/A)}{X_*(S_\psc)} \longisoto
\frac{X_*({}^gS/{}^gA)}{X_*({}^gS_\psc)}
\]
induced by conjugation by $g$ is $\Gamma$-equivariant.
\end{Lemma}

\begin{proof}
By assumption, the difference between the maps induced by $g$ and ${}^\sigma g$
is an element of $(Z_G^\circ(A)\cap N_G(S))(\ov{F})$,
and by Lemma~\ref{trivaction},
this group acts trivially on $X_*(S/A)/X_*(S_\psc)$.
So the two maps must be equal.
\end{proof}

\begin{Definition}\label{Ydef}
Let $A$ be a finite subgroup of~$G$ contained in a maximal $F$-torus.
\begin{enumerate}
\item Given a maximal torus $S$ of~$G$ containing $A$, define
\[
\wt{Y}_{+,\text{tor}}(A\to G;S) \defeq
\varinjlim_{E/F} \frac{(X_{*}(S'/A')/X_{*}(S'_{\psc}))^{N}}%
{I(X_{*}(S')/X_{*}(S'_{\psc}))},
\]
taking the direct limit, as in \eqref{linalgfuncbas},
over all finite extensions $E/F$ splitting~$S'$.

\item Define
$\wt{Y}_{+,\text{tor}}(A\to G) \defeq 
\varinjlim_{A'\subseteq S'\subseteq G} \wt{Y}_{+,\tn{tor}}(A'\to G,S')$,
taking the direct limit over the connected groupoid containing $(S,A)$
whose objects are pairs $(S',A')$ and whose morphisms from $(S',A')$ to $(S'',A'')$
are the $g\in G(\ov{F})$ such that ${}^g(S',A') = (S'',A'')$ and
$dg\in Z^1(F,Z_G^\circ(A')\cap N_G(S'))$.
\end{enumerate}
\end{Definition}

That is, by Lemma~\ref{stabisom},
any $g\in G(\ov{F})$ as in Definition~\ref{Ydef}
induces $\Gamma$-equivariant isomorphisms
\[
\frac{X_*(S'/A')}{X_*(S'_\psc)} \longisoto
\frac{X_*(S''/A'')}{X_*(S''_\psc)},
\qquad
\frac{X_*(S')}{X_*(S'_\psc)} \longisoto
\frac{X_*(S''/A'')}{X_*(S''_\psc)},
\]
whence an isomorphism
$\wt{Y}_{+,\tn{tor}}(A'\to G;S')\isoto
\wt{Y}_{+,\tn{tor}}(A''\to G;S'')$
that we use to form the outer direct limit in Definition~\ref{Ydef}.
The groupoid in the definition is connected
because any two maximal tori containing $A$
are conjugate by some $g\in Z_G^\circ(A)(\ov{F})$,
which thus satisfies $dg\in Z^1(F,Z_G^\circ(A)\cap N_G(S))$.
Moreover, $\wt{Y}_{+,\tn{tor}}$ agrees with $\ov{Y}_{+,\tn{tor}}$
on $\mathcal{R}_\tn{cent}$,
since in that case $Z_G(A) = G$
and thus the outer limit there passes over all
$F$-rational maximal tori in $G$.

\begin{Lemma}
Let $S$ be a maximal torus containing~$A$.
Then there is a canonical isomorphism
\[
\wt{Y}_{+,\tn{tor}}(A\to G)\simeq
\frac{\wt{Y}_{+,\tn{tor}}(A\to G;S)}{W(G,Z_G^\circ(A))(F)}.
\]
\end{Lemma}

\begin{proof}
By definition of $\wt{Y}_{+,\tn{tor}}(A\to G)$,
this set is isomorphic to the orbits
for the action on the set $\wt{Y}_{+,\tn{tor}}(A\to G;S)$
of the group of $n\in N_G(A,S)(\ov{F})$
such that $dn\in Z^1(F,Z_G^\circ(A)\cap N_G(S))$.
But since the group $(Z_G^\circ(A)\cap N_G(S))(\ov{F})$
acts trivially on this set by Lemma~\ref{trivaction},
the action descends to an action of
\[
\biggl(\frac{N_G(A,S)(\ov{F})}{(Z_G^\circ(A)\cap N_G(S))(\ov{F})}\biggr)^\Gamma
\]
which is isomorphic to $W(G,Z_G^\circ(A))(F)$.
\end{proof}

\begin{Proposition} \label{Yfunctorial}
The set $\wt{Y}_{+,\tn{tor}}(A\to G)$ is functorial in $[A\to G]$
for morphisms in $\mathcal{R}$.
\end{Proposition}

\begin{proof}
Fix a morphism $\phi\colon [A_{1} \to G_{1}] \to [A_{2} \to G_{2}]$ in $\mathcal{R}$.
The goal is to assign a morphism of pointed sets between
$\wt{Y}_{+,\text{tor}}(A_{1} \to G_{1})$
and $\wt{Y}_{+,\text{tor}}(A_{2} \to G_{2})$.
For brevity, write $H_j \defeq Z_G^\circ(A_j)$.

Let $S_j$ be maximal tori of $G_j$ containing $A_j$ for $j=1,2$
and such that $\phi(S_1)\subseteq S_2$.
The map $\phi$ lifts uniquely to a map $G_{1,\Sc} \to G_{2,\Sc}$,
then restricts to a map $S_{1,\psc}\to S_{2,\psc}$
which induces a map $\wt{Y}_{+,\tn{tor}}(A_1\to G_1;S_1)\to
\wt{Y}_{+,\tn{tor}}(A_2\to G_2;S_2)$.
To show that this map descends to the direct limit defining 
$\wt{Y}_{+,\tn{tor}}(A\to G)$,
we need to show that for any $g_{1} \in G_1(\ov{F})$
with $dg_1\in Z^1(F,N_{H_1}(S_1))$,
there is $g_2\in G_2(\ov{F})$ with $dg_2\in Z^1(F,N_{H_2}(S_2))$
making the following diagram commute:
\[
\begin{tikzcd}
\wt{Y}_{+,\tn{tor}}(A_1\to G_1;S_1) \rar{\phi}\dar{g_1} &
\wt{Y}_{+,\tn{tor}}(A_2\to G_2;S_2) \dar{g_2} \\
\wt{Y}_{+,\tn{tor}}({}^{g_1}A_1\to G_1;{}^{g_1}S_1) \rar{\phi} &
\wt{Y}_{+,\tn{tor}}({}^{g_2}A_2\to G_2;{}^{g_2}S_2).
\end{tikzcd}
\]
Naively one would like to take $g_2=\phi(g_1)$,
but this choice may not work: although $\phi(H_1)
\subseteq Z_{G_2}^\circ(\phi(A_1)) = H_2$ by assumption,
implying that $d\phi(g_1)\in Z^1(F,H_2^\circ)$,
since $\phi(N_{H_1}(S_1))\not\subseteq N_{H_2}(S_2)$ in general,
the torus ${}^{\phi(g_1)}S_2$ need not be defined over~$F$.
To remedy this, since all maximal tori of $Z_{G_2}^\circ(A_2)$
are conjugate over $\ov{F}$,
we may choose $h \in H_2(\ov{F})$
such that ${}^{\phi(g_1)h}S_2$ is defined over~$F$.
Let $g_2=\phi(g_1)h$,
so that $dg_2\in Z^1(F,N_{H_2}(S_2))$.
The resulting map $\wt{Y}_{+,\tn{tor}}(A_2\to G_2;S_2)
\to \wt{Y}_{+,\tn{tor}}({}^{g_2}A_2\to G_2;{}^{g_2}S_2)$
is independent of the choice of $h$,
and thus the diagram above commutes
and the maps there are independent of all choices.
\end{proof}

Using Proposition~\ref{Yfunctorial},
it will be convenient to slightly extend the definition of
$\wt{Y}_{+,\tn{tor}}(A\to G)$ as follows:
Given a possibly infinite subgroup $B\subseteq G$ contained in a maximal torus, let
\begin{equation}\label{infB}
\wt{Y}_{+,\tn{tor}}(B\to G) \defeq
\varinjlim_{A\subseteq B} \wt{Y}_{+,\tn{tor}}(A \to G),
\end{equation}
where the direct limit is taken over finite subgroups $A\subseteq B$
such that $Z_G^\circ(A) = Z_G^\circ(B)$.

Our definition of the target of the Kottwitz map is not quite complete
since $\wt{Y}_{+,\tn{tor}}(A\to G)$ depends on~$A$
but $H^1_\tn{reg}(\cE,G)$ does not.
To remove the dependence on~$A$,
we use the functoriality of Proposition~\ref{Yfunctorial}
to form a direct limit.

If we restrict the category $\mathcal{R}$ to pairs $[A \to G]$
where $G$ is fixed and $A$ is a finite $F$-rational subgroup contained in a torus,
then there is a morphism $[A' \to G] \to [A \to G]$
if and only if $A' \subseteq A$ and $Z_G^\circ(A) = Z_G^\circ(A')$.
These pairs $[A\to G]$ thus form a subcategory $\mathcal{R}_{(G)}$
equal to the disjoint union of categories
\[
\mathcal{R}_{(G)} = \bigsqcup_{[H]} \mathcal{R}_{(G)}^{H},
\]
where $\mathcal{R}_{(G)}^{H}$ is the subcategory of all $[A \to G]$
with $Z_{G}(A) = H$ and $[H]$ ranges over all stable conjugacy classes
of full-rank reductive $F$-subgroups of~$G$,
in other words, the possible centralizers
of finite multiplicative subgroups of~$G$
that lie in an $F$-torus.
Of course, in our application to $H^1_\tn{reg}(\cE,G)$,
only those $H$ which are rigid Newton centralizers will be relevant.

\begin{Definition}
$\displaystyle\wt{Y}_{+,\tn{tor}}(G) \defeq \bigsqcup_{[H]} \varinjlim_A
\wt{Y}_{+,\tn{tor}}(A\to G)$, where
\begin{itemize}
\item
$[H]$ ranges over the stable conjugacy classes of
full-rank reductive $F$-subgroups of~$G$,
\item
$A$ ranges over the poset of finite multiplicative $F$-subgroups $A$ of~$G$
with $H = Z_G^\circ(A)$, and
\item
the (injective) transition map for $A\subseteq A'$
is induced by the functoriality of Proposition~\ref{Yfunctorial}.
\end{itemize}
\end{Definition}

\noindent Note that each set in the disjoint union is independent of the choice of~$H$
since $\wt{Y}_{+,\tn{tor}}(A\to G) = \wt{Y}_{+,\tn{tor}}(A'\to G)$
whenever $Z_G^\circ(A)$ is stably conjugate to~$Z^\circ_G(A')$. 

To ensure that $\wt{Y}_{+,\tn{tor}}$
is functorial with respect to all maps of reductive groups,
we need to shrink it slightly.
This reduction is effected by the following avatar of the Newton map.

\begin{Lemma} \label{lambdalemma}
The homomorphisms
$\wt{Y}_{+,\tn{tor}}(A\to G;S) \longrightarrow X_*(S/A)/X_*(S)$
induce a map 
\[
\lambda\colon\wt{Y}_{+,\tn{tor}}(G)\to\cC_\tn{st}(G),
\]
which is natural with respect to morphisms in $\mathcal{R}$.
\end{Lemma}

\begin{proof}
This claim is a straightforward consequence
of the easily-verified fact that the family of given homomorphisms
is compatible with the equivalence relations
defining $\wt{Y}_{+,\tn{tor}}$ and $\cC_\tn{st}$.
We leave the details to the reader.
\end{proof}

\begin{Definition}
\begin{enumerate}
\item
For $M$ a rigid Newton centralizer,
$\ov{Y}_{+,\tn{tor}}(G)_M\defeq \lambda^{-1}(\cC_\tn{st}(G)_M)$.
\item $\displaystyle\ov{Y}_{+,\tn{tor}}(G) \defeq \bigsqcup_{[M]}\ov{Y}_{+,\tn{tor}}(G)_M$,
where $[M]$ ranges over stable conjugacy classes
of rigid Newton centralizers $M$.
\end{enumerate}
\end{Definition}

\begin{Proposition}
$G\mapsto\ov{Y}_{+,\tn{tor}}(G)$
is functorial for homomorphisms of reductive $F$-groups.
\end{Proposition}

\begin{proof}
Let $f\colon G \to G'$ be a morphism of connected reductive groups.
The goal is to define a map
\begin{equation*}
\ov{Y}_{+,\tn{tor}}(f) \colon
\bigsqcup_{[M]} \ov{Y}_{+,\tn{tor}}(G)_M
\longrightarrow \bigsqcup_{[M']} \ov{Y}_{+,\tn{tor}}(G')_{M'}
\end{equation*}
Fix a rigid Newton centralizer $M$ of $G$.
We define the map on the $M$-block above as follows:
for each $\bar x\in\ov{Y}_{+,\tn{tor}}(G)_M$,
choose a maximal torus $T$ of~$M$, a subgroup $A\subseteq Z(M)$, 
and a representative
$x\in \wt{Y}_{+,\tn{tor}}(A\to G;T)$.
Let $\lambda(x)$ be a representative of~$\lambda(\bar x)$.
Using Lemma~\ref{newtonmaptarget},
we interpret $\lambda(x)$ as a homomorphism $u\to G$ defined over~$F$,
which we choose to factor through~$A$,
and we let $M'_x\defeq Z_{G'}^\circ(f\circ\lambda(x))$.
Choosing $A$ minimally, we may assume that $\lambda(x)(u) = A$.
The stable conjugacy class of $M'_x$ depends only on~$\bar x$,
and $f(M)\subseteq M'_x$ because $M = Z_G^\circ(\lambda(x))$ by assumption.
Enlarging $f(T)$ to a maximal torus $T'$ of~$G'$ and setting $A'=f(A)$,
we may interpret $f(x)$ as an element of the set
$\wt{Y}_{+,\tn{tor}}(A'\to G';T')$,
which then maps to $\ov{Y}_{+,\tn{tor}}(G')_{M'_x}$.
The assignment $x\mapsto f(x)$
is evidently compatible with enlarging $A$ while preserving $M = Z_G^\circ(A)$.
Moreover, the equivalence class of $f(x)$ in $\ov{Y}_{+,\tn{tor}}(G')$
is independent of the choice of~$x$:
for $g\in G(\ov{F})$,
if $dg\in Z^1(F,M)$, then $f(g)\in Z^1(F,M'_x)$.
\end{proof}

\subsection{Defining the map}\label{kottdefsec}
The goal of this subsection is to extend the Kottwitz map
of functors on $\mathcal{R}_\tn{cent}$ from Theorem \ref{Tatenak1}
to a morphism $H^{1}(\cE, -)\to\ov{Y}_{+,\tn{tor}}(-)$
of functors on the category $\mathcal{R}$.

\begin{Proposition}\label{keyprop1}
Let $A$ be a finite subgroup of~$G$
contained in a torus and let $M=Z_G^\circ(A)$.
There exists a unique dashed arrow making the following diagram commute:
\[
\begin{tikzcd}
H^1(\cE,A\to M)_\Greg \rar[twoheadrightarrow]\dar{\iota_{[A\to M]}} &
H^1_\tn{reg}(\cE,A\to G)_M \dar[dashed] \\
\wt{Y}_{+,\tn{tor}}(A\to M) \rar &
\wt{Y}_{+,\tn{tor}}(A\to G).
\end{tikzcd}
\]
\end{Proposition}

\begin{proof}
Choose representatives $x,y \in Z^{1}(\cE, A \to M)$
such that $Z_G^\circ(x|_u) = Z_G^\circ(y|_u) = M$.
By assumption there is some $g \in G(\ov{F})$ such that $y = {}^gx$.
Then $g \in N_{G}(M)(\ov{F})$, since 
$ \prescript{g}{}M = \prescript{g}{}Z_G^\circ(x|_u)
= Z_G^\circ(\prescript{g}{}x|_u) = Z_G^\circ(y|_u) = M$,
and in fact $g$ projects to an element of $W(G,M)(F)$, since $dg\in Z^1(F,M)$.
In other words, in terms of the natural action of $W(G,M)(F)$
on $H^1(\mc{E},A\to M)$ recalled in \eqref{Weyltransaction},
the classes $[x]$ and $[y]$ lie in the same $W(G,M)(F)$-orbit.

Using \cite[Corollary 3.7]{Kaletha16a},
we may replace $y$ by a cohomologous element in $Z^{1}(\mc{E}, A \to M)$
in order to assume further that $y \in Z^{1}(\cE, A \to S)$,
where $S$ is a fixed elliptic maximal torus of $M$.
The existence of the dashed arrow will follow if we can show that
$\iota_{[A \to M]}([x])$ and $\iota_{[A \to M]}([y])$
have the same image under the map
\begin{equation}\label{funckeymap}
\wt{Y}_{+,\tn{tor}}(A \to M;S)
\longrightarrow \frac{\wt{Y}_{+,\tn{tor}}(A\to G;S)}{W(G,M)(F)},
\end{equation}
in which case one can define the dashed map
by choosing a lift to $H^1(\cE,A\to M)_\Greg$
and sending it through the bottom left corner of the diagram,
all choices of lift yielding the same result.

We claim that the composition 
\begin{equation}\label{Weylequivcomp}
H^{1}(\mc{E}, A \to M) \xrightarrow{\iota_{[A \to M]}}
\wt{Y}_{+,\tn{tor}}(A\to M;S) \longrightarrow
\wt{Y}_{+,\tn{tor}}(A\to G;S)
\end{equation}
sends $W(G,M)(F)$-orbits to $W(G,M)(F)$-orbits, where $W(G,M)(F)$ acts on the right-hand side of \eqref{Weylequivcomp} by lifting elements to $N_{G}(S)(\ov{F})$ and then acting as usual. This claim would give the desired result since $[x]$ is a $W(G,M)(F)$-translate of $[y]$.

Given $w \in W(G,M)(F)$, let $\dot w \in N_G(M)(\ov{F})$ be an arbitrary lift.
Then $d\dot w \in Z^{1}(F, M)$, and since $S$ is elliptic,
the class of $d\dot w$ in $H^{1}(F, M)$
is cohomologous to a class in the image of $H^{1}(F,S)$.
So we may replace $\dot w$ by an
$M(\ov{F})$-translate to assume that $d\dot w \in Z^{1}(F, S)$,
and take $G'$ to be the twisted form of $G$ determined by the cocycle $d\dot w$.
In particular, we have an $F$-rational isomorphism $\Inn(\dot w)\colon G \to G'$ which restricts to an $F$-rational isomorphism $M \isoto M' \defeq \prescript{d\dot w}{}M$. By construction, the composition $S \to M \isoto M'$ is still the embedding of an $F$-rational maximal torus, so we can and do view $S$ as a maximal torus of $M'$.%

Given $x \in Z^{1}(\cE, A \to M)$,
consider the cocycle $x' = {}^{\dot w^{-1}}x$
representing the translate of $[x]$ by~$w^{-1}$.
To analyze $x'$, we will use the commutative diagram of bijections
\begin{equation}\label{kotttwistdiag}
\begin{tikzcd}[column sep=2.5cm]
H^{1}(\mc{E}, A \to M) \arrow["\text{Ad}(\dot w)"]{r} \arrow["\iota_{[A \to M]}"]{d} &
H^{1}(\mc{E}, A \to M') \arrow["\cdot d\dot w"]{r} \arrow["\iota_{[A \to M']}"]{d} &
H^{1}(\mc{E}, A \to M) \arrow["\iota_{[A \to M]}"]{d} \\
\wt{Y}_{+,\tn{tor}}(A\to M;S) \arrow["\text{Ad}(\dot w)"]{r} &
\wt{Y}_{+,\tn{tor}}(A\to M;S) \arrow["+\iota_{[A \to M]}(d\dot w)"]{r} &
\wt{Y}_{+,\tn{tor}}(A\to M;S)
\end{tikzcd}
\end{equation}
where we use that $\wt{Y}_{+,\tn{tor}}(A\to M;S)$
is unaffected by twisting the $\Gamma$-action on $M(\ov{F})$
by $d\dot w$, since $d\dot w$ takes values in $S$.
The lefthand square commutes by functoriality of $\iota$
and the righthand square commutes by \cite[Lemma 1.4]{Kottwitz86}.

The lefthand square in \eqref{kotttwistdiag} gives
$\iota_{[A \to M]}([x']) = \Inn(\dot w^{-1})\iota_{[A \to M']}([\dot wx'\dot w^{-1}])$
and one computes easily that $\dot wx'\dot w^{-1} = x \cdot (d\dot w)^{-1}$,
so that
\[
\iota_{[A \to M]}([x']) = \Inn(\dot w^{-1})\iota_{[A \to M']}([x\cdot(d\dot w)^{-1}])
\]
The commutativity condition for the righthand square in \eqref{kotttwistdiag}
applied to $[x\cdot(d\dot w)^{-1}]$ gives
\[
\iota_{[A \to M']}([x\cdot(d\dot w)^{-1}]) + \iota_{[A\to M]}(d\dot w) = \iota_{[A\to M]}([x]).
\]
Combining these expressions, we find that
\[
\iota_{[A\to M]}([x']) = \Inn(\dot w^{-1})(\iota_{[A\to M]}([x])
- \iota_{[A\to M]}([d\dot w])).
\]
The last step uses the identity
\begin{equation}\label{newkeyeqn}
\Inn(\dot w^{-1})\iota_{[A \to M]}([d\dot w]^{-1}) = \iota_{[A \to M]}([d(\dot w^{-1})]),
\end{equation}
which we will prove momentarily.
Granted \eqref{newkeyeqn}, we see that
\[
\iota_{[A\to M]}([x']) = \Inn(\dot w^{-1})(\iota_{[A\to M]}([x])
+ \iota_{[A\to M]}([d(\dot w^{-1})])).
\]
This shows that \eqref{Weylequivcomp}
preserves $W(G,M)(F)$-orbits,
as $\iota_{[A \to M]}([d(\dot w^{-1})])$ lies in the kernel of 
$\wt{Y}_{+,\tn{tor}}(A\to M;S) \longrightarrow \wt{Y}_{+,\tn{tor}}(A\to G;S)$.

To prove \eqref{newkeyeqn},
let $\xi \in Z^{2}(F,u)$
be the representative of the canonical class corresponding to $\cE$
described in \cite[\S~4.5]{Kaletha16a} and used there to construct
the Tate--Nakayama isomorphism,
By the definition of that isomorphism,
$\xi \cup \iota_{[A \to M]}([d\dot w]) = [d\dot w]$.
One then computes that
\begin{equation*}
\xi \cup (\text{Ad}(\dot w^{-1})\iota_{[A \to M]}([d\dot w]))
=  (e \mapsto {}^e\dot w^{-1}d\dot w {}^e\dot w)
= [d(\dot w^{-1})]^{-1},
\end{equation*}
or equivalently, that
$\iota_{[A \to M]}([d(\dot w^{-1})]^{-1})
= \text{Ad}(\dot w^{-1})\iota_{[A \to M]}([d\dot w])$,
and \eqref{newkeyeqn} now follows by taking the inverse of both sides,
since $\iota_{[A \to M]}$ is a group homomorphism.
\end{proof}

\begin{Warning}
Theorem \ref{Tatenak1} gives $H^{1}(\cE, A \to M)$
the structure of an abelian group,
but the action of $W(G,M)(F)$ does not preserve this structure.
For example, the neutral element can move since
the class $[d\dot w]\in H^{1}(\mc{E}, A \to M)$ is often nontrivial.
\end{Warning}

We can use Proposition~\ref{keyprop1} to define a candidate
$\tilde\kappa\colon H^1_\tn{reg}(\cE,G)\to\wt{Y}_{+,\tn{tor}}(G)$
for the rigid Kottwitz map,
namely, the unique map such that for every rigid Newton centralizer $M$
and finite subgroup $A\subseteq Z(M)$,
the restriction of $\tilde\kappa$ to $H^1_\tn{reg}(\cE,A\to G)_M$ is the composite map
\[
H^1_\tn{reg}(\cE,A\to G)_M \to \wt{Y}_{+,\tn{tor}}(A\to G)
\into \wt{Y}_{+,\tn{tor}}(G),
\]
where the left map is the dashed arrow of Proposition~\ref{keyprop1}.
To finish the construction of the Kottwitz map,
we will prove several compatibility conditions showing that
$\tilde\kappa$ factors through $\ov{Y}_{+,\tn{tor}}(G)$.

\begin{Lemma} \label{lemtorussquare}
Let $T\subseteq G$ be a maximal torus.
The following square commutes:
\[
\begin{tikzcd}[row sep=small]
H^1(\cE,T) \rar\dar[near start]{\tilde\kappa} &
H^1_\tn{reg}(\cE,G) \dar[near start]{\tilde\kappa} \\
\wt{Y}_{+,\tn{tor}}(T) \rar &
\wt{Y}_{+,\tn{tor}}(G).
\end{tikzcd}
\]
\end{Lemma}

\begin{proof}
Let $x\in Z^1(\cE,T)$, let $A=x(u)$,
and let $M=Z_G^\circ(A)$.
Consider the following diagram
\[
\begin{tikzcd}[row sep=small]
H^1(\cE,A\to T)_M \rar \arrow[rr,bend left=10] \dar[near start]{\tilde\kappa} &
H^1_\tn{reg}(\cE,A\to G)_M \dar[near start]{\tilde\kappa} &
H^1(\cE,A\to M)_\Greg \lar[twoheadrightarrow] \dar[near start]{\tilde\kappa} \\
\wt{Y}_{+,\tn{tor}}(A\to T) \rar \arrow[rr,bend right=10] &
\wt{Y}_{+,\tn{tor}}(A\to G) &
\wt{Y}_{+,\tn{tor}}(A\to M). \lar
\end{tikzcd}
\]
Here $H^1(\cE,A\to T)_M$ denotes the classes of $x\in Z^1(\cE,A\to T)$
such that $Z_G^\circ(x|_u)=M$.
By a diagram chase,
it suffices to show commutativity of all faces in the diagram
except possibly the left square.
The top triangle commutes by functoriality of group cohomology.
The bottom triangle commutes by functoriality of $\wt{Y}_{+,\tn{tor}}$
for morphisms in $\mathcal{R}$ (Proposition~\ref{Yfunctorial}).
The right square commutes by definition of~$\tilde\kappa$.
The back rectangle commutes by naturality
of the basic Kottwitz map \cite[Theorem~4.11]{Kaletha16a}.
\end{proof}

\begin{Lemma} \label{kottnewtcompat}
Let $\lambda$ be the map of Lemma~\ref{lambdalemma}.
The following triangle commutes:
\[
\begin{tikzcd}
H^1_\tn{reg}(\cE,G) \rar{\tilde\kappa} \arrow[rr,bend left=20,"\nu" near end] &
\wt{Y}_{+,\tn{tor}}(G) \rar{\lambda} &
\cC_\tn{st}(G).
\end{tikzcd}
\]
\end{Lemma}

\begin{proof}
Given $[x]\in H^1_\tn{reg}(\cE,G)$,
choose a maximal torus $T$ of~$G$
such that $[x]$ lies in the image of $H^1(\cE,T)$.
Consider the following diagram:
\[
\begin{tikzcd}
H^1_\tn{reg}(\cE,T) \dar\rar{\tilde\kappa} \arrow[rr,bend left=20,"\nu" near end] &
\wt{Y}_{+,\tn{tor}}(T) \rar{\lambda} &
\cC_\tn{st}(T) \dar \\
H^1_\tn{reg}(\cE,G) \rar{\tilde\kappa} \arrow[rr,bend left=20,"\nu" near end] &
\wt{Y}_{+,\tn{tor}}(G) \arrow[from=u,crossing over]\rar{\lambda} &
\cC_\tn{st}(G).
\end{tikzcd}
\]
By a diagram chase, it suffices to show commutativity
of the other four faces of the diagram above,
excluding the bottom triangle.
The upper triangle commutes by Lemma~\ref{Lemtncompat}.
The back square commutes by naturality of the Newton map.
The left square commutes by Lemma~\ref{lemtorussquare}.
Finally, a short direct argument shows that the right square commutes.
\end{proof}

We can now use the Newton decomposition to define the rigid Kottwitz map.

\begin{Definition} \label{kottdef}
The \textit{rigid Kottwitz map} is the unique map
$\kappa\colon H^1_\tn{reg}(\cE,G)\to\ov{Y}_{+,\tn{tor}}(G)$
such that for every rigid Newton centralizer $M$
and finite subgroup $A\subseteq Z(M)$,
the restriction of $\kappa$ to $H^1_\tn{reg}(\cE,G)_M$ is the composite map below,
in which the left map is the dashed arrow of Proposition~\ref{keyprop1},
which factors through $\ov{Y}_{+,\tn{tor}}(A\to G)$ by Lemma~\ref{kottnewtcompat}:
\[
H^1_\tn{reg}(\cE,A\to G)_M \to \ov{Y}_{+,\tn{tor}}(A\to G)
\into \ov{Y}_{+,\tn{tor}}(G).
\]
\end{Definition}

In this language, Lemma~\ref{kottnewtcompat}
implies that the Newton map factors through the Kottwitz map.
in the sense that the following diagram commutes,
where $\lambda$ is the map of Lemma~\ref{lambdalemma}:
\begin{equation}\label{newtonkottwitz}
\begin{tikzcd}
H^1_\tn{reg}(\cE,G) \rar{\kappa} \drar{\nu} &
\ov{Y}_{+,\tn{tor}}(G) \dar{\lambda} \\
& \cC_\tn{st}(G).
\end{tikzcd}
\end{equation}

\begin{Proposition} \label{kottinj}
The Kottwitz map $\kappa\colon H^1_\tn{reg}(\cE,G)\to\ov{Y}_{+,\tn{tor}}(G)$
is injective and natural in~$G$.
\end{Proposition}

\begin{proof}
First, we will show that $\kappa$ is injective.
Let $a,b\in H^1_\tn{reg}(\cE,G)$
and suppose that $\kappa(a) = \kappa(b)$.
Then $\nu(a) = \nu(b)$ as well, by \eqref{newtonkottwitz}.
Let $x,y\in Z^1_\tn{reg}(\cE,G)$ be regular representatives of $a$ and $b$.
The assumption that $\nu(a) = \nu(b)$
implies that we may choose $x$ and~$y$ with $x|_u = y|_u$.
Let $M\defeq Z_G^\circ(x|_u) = Z_G^\circ(y|_u)$.
It will also be useful to replace $x$ with an $M$-cohomologous
representative in order to assume that $x \in Z^{1}(\cE, A \to S)$
for an $F$-rational maximal torus $S$ of $M$.

Write $G_\cE$ for $G(\ov{F})$ as an $\cE$-module
with action inflated from the usual $\Gamma$ action.
The cocycles $x$ and $y$ give twisted forms $^{x}G$ and $^{y}G$
of $G_{\mc{E}}$ and, via taking the images $\bar x$ and $\bar y$
of $x$ and $y$ in $Z^{1}(\Gamma, M_{\tn{ad}}(\ov{F}))$,
twisted forms $^{\bar{x}}M$ and $\prescript{\bar{y}}{}M$ of $M$ such that
$(\prescript{\bar{x}}{}M)_{\mc{E}} \times^{M_{\mc{E}}} G_{\mc{E}} = \prescript{x}{}G$
and similarly for $y$.
For a group scheme $N$ over $\ov{F}$ on which $\cE$
acts by algebraic automorphisms along with an algebraic group monomorphism
$A_{\ov{F}} \to Z(N)$ compatible with the $\cE$-action on both schemes,
denote by $H^{1}(\cE, A \to N)$ the subset of $H^{1}(\cE, N(\ov{F}))$
consisting of classes $[c]$ which have a representative $c$ such that
$c|_{u(\ov{F})}$ is a morphism of algebraic groups from $u(\ov{F})$ to $A(\ov{F})$,
which is necessarily defined over $F$.

Take the images of $[x], [y] \in H^{1}(\cE, A \to M)$ in
$H^{1}(\cE, A \to \prescript{\bar{x}}{}M)$
under the standard twisting map $[c] \mapsto [cx^{-1}]$
and denote them by $[x]'$ and $[y]'$,
so that $[x]'$ is, by design, the neutral class.
These elements lie in the subset
$H^{1}(F, \prescript{\bar{x}}{}M) \subseteq H^{1}(\cE, A \to \prescript{\bar{x}}{}M)$
since their restriction to $u$ is trivial.
This twisting bijection reduces the problem to showing that the image of $[y]'$
in $H^{1}(\cE, A \to \prescript{x}{}G)$, or, equivalently, in
$H^{1}(\cE, A \to \prescript{x}{}N_{G}(M))$, is the neutral class.

Define $H^{1}(\cE, A \to \prescript{x}{}N_{G}(M))^{\circ}$
to be the image of the map
\[
H^{1}(\cE, A \to \prescript{\bar{x}}{}M)
\longrightarrow H^{1}(\cE, A \to \prescript{x}{}N_{G}(M))
\]
so that, by construction, the image of 
$[y]' \in H^{1}(\cE, A \to \prescript{\bar{x}}{}M)$
in $H^{1}(\cE, A \to \prescript{x}{}N_{G}(M))$
lies in $H^{1}(\cE, A \to \prescript{x}{}N_{G}(M))^{\circ}$.
The short exact sequence of nonabelian $\cE$-modules
\begin{equation*}
1 \to \prescript{\bar{x}}{}M  \to \prescript{x}{}N_{G}(M) \to W(G,M) \to 1
\end{equation*}
yields the exact sequence
\begin{equation*}
W(G,M)(F) \to
H^{1}(\cE, A \to \prescript{\bar{x}}{}M) \to
H^{1}(\cE, A \to \prescript{x}{}N_{G}(M))^{\circ} \to 1,
\end{equation*}
Here we have deliberately omitted the twist from the rightmost term
of the short exact sequence and the leftmost term of the above sequence
because the $\cE$-action is inflated from the usual $\Gamma$-action.
Therefore, using the action of $W(G,M)(F)$
defined in \eqref{Weyltransaction} and the discussion surrounding it,
\[
H^{1}(\cE, A \to \prescript{x}{}N_{G}(M))^{\circ}
= H^{1}(\cE, A \to \prescript{\bar{x}}{}M)/W(G,M)(F).
\]
Hence $\iota_{[A \to M]}$ induces a pointed bijection
\begin{equation*}
H^{1}(\cE, A \to \prescript{x}{}N_{G}(M))^{\circ}
\isoto  \ov{Y}_{+,\tn{tor}}(A \to \prescript{\bar{x}}{}M)/W(G,M)(F).
\end{equation*} 

Finally, one uses the commutative diagram
(cf. \eqref{kotttwistdiag} from the proof of Proposition \ref{keyprop1})
\begin{equation*}
\begin{tikzcd}[column sep=2.5cm]
H^{1}(\mc{E}, A \to M) \arrow["\cdot x^{-1}"]{r} \arrow["\iota_{[A \to M]}"]{d} & H^{1}(\mc{E}, A \to \prescript{\bar{x}}{}M) \arrow["\iota_{[A' \to \prescript{\bar{x}}{}M]}"]{d}  \\
\wt{Y}_{+,\tn{tor}}(A\to M;S)
\arrow["+\iota_{[A \to M]}(x^{-1})"]{r} &
\wt{Y}_{+,\tn{tor}}(A\to M;S)
\end{tikzcd}
\end{equation*}
along with the commutative diagram
\begin{equation*}
\begin{tikzcd}
H^{1}(\cE, A \to N_{G}(M))^{\circ} \arrow{r} & H^{1}(\cE, A \to \prescript{x}{}N_{G}(M))^{\circ} \arrow["\sim"]{r} & \ov{Y}_{+,\tn{tor}}(A \to \prescript{\bar{x}}{}M)/W(G,M)(F) \\\
H^{1}(\cE, A \to M) \arrow{u} \arrow{r} & H^{1}(\cE, A \to \prescript{\bar{x}}{}M) \arrow{u} \arrow["\sim"]{r} & \ov{Y}_{+,\tn{tor}}(A \to \prescript{\bar{x}}{}M) \arrow{u},
\end{tikzcd}
\end{equation*} 
where the left maps in both rows are twisting by $[x^{-1}]$,
to deduce that the image of $[y]'$ is trivial in
$\ov{Y}_{+,\tn{tor}}(A \to \prescript{\bar{x}}{}M)/W(G,M)(F)$
and is thus the neutral class in $H^{1}(\cE, A \to \prescript{x}{}N_{G}(M))$.

Finally, to prove that $\kappa$ is natural,
let $f\colon G\to G'$ be a homomorphism of reductive groups,
and consider the cube below, where
\begin{enumerate}
\item $H^1_\tn{bas}(\cE,M)_{M'}$ is classes represented by 
$x\in Z^1_\tn{bas}(\cE,M)$ with $Z_{G'}^\circ(f\circ x|_u) = M'$,
\item $H^1_\tn{reg}(\cE,G)_{M,M'}$ is classes represented by 
$x\in Z^1_\tn{reg}(\cE,G)$ with 
$Z_G^\circ(x|_u)=M$ and $Z_{G'}^\circ(f\circ x|_u)=M'$, and
\item $\ov{Y}_{+,\tn{tor}}(M)_{M'}$ is the preimage of $\ov{Y}_{+,\tn{tor}}(M')_{M'}$:
\end{enumerate}
\[
\begin{tikzcd}[column sep=small, row sep=small]
H^1_\tn{bas}(\cE,M)_{M'} \drar\arrow{rr}\arrow{dd} &&
H^1_\tn{bas}(\cE,M')_{G'\textnormal{-reg}} \drar\arrow{dd} & \\
& H^1_\tn{reg}(\cE,G)_{M,M'} \arrow[rr,crossing over] &&
H^1_\tn{reg}(\cE,G')_{M'} \arrow{dd} \\
\ov{Y}_{+,\tn{tor}}(M)_{M'} \drar\arrow{rr} &&
\ov{Y}_{+,\tn{tor}}(M')_{M'} \drar \\
& \ov{Y}_{+,\tn{tor}}(G) \arrow{rr}\arrow[from=uu,crossing over] &&
\ov{Y}_{+,\tn{tor}}(G') \\
\end{tikzcd}
\]
By a diagram chase, it suffices to show the commutativity
of all faces except possibly the front face.
The top face commutes by functoriality of $H^1(\cE,-)$.
The bottom face commutes by functoriality of $\ov{Y}_{+,\tn{tor}}$.
The left and right faces commute by definition of~$\tilde\kappa$.
The back face commutes by naturality of the basic Kottwitz map
\cite[Theorem~4.11]{Kaletha16a}.
\end{proof}

\section{Dual interpretations of $\mc{E}$-cohomology}\label{dualEsec}

\subsection{$L$-embeddings from twisted Levi subgroups}\label{Lembsubsec}
This subsection reviews Kaletha's $L$-groups ${}^LM_\pm$
and their canonical $L$-embeddings ${}^LM_\pm\to{}^LG$,
following \cite{Kaletha19b}.
These $L$-groups have a dual interpretation
in terms of a certain double cover $M(F)_\pm$ of~$M(F)$,
which we will come to in Section~\ref{basicdoublesubsec}.

Kaletha's theory constructs, for any twisted Levi subgroup $M$ of~$G$,
a possibly non-split extension ${}^LM_\pm$ of $\Gamma$ by~$\widehat M$
as well as a $\widehat G$-conjugacy class of admissible embeddings ${}^LM_\pm\to{}^LG$.
The precise definitions of this extension and its $L$-embedding
are designed to be explicit and independent of all choices.
We refer the reader to \cite{Kaletha19b} for full details
and Lemma~\ref{lem:lgrouppullback} for a poor man's version.

We start by explaining how to pass twisted Levi subgroups to the Galois side,
an idea implicit in \cite{Kaletha19b} that we spell out for completeness.

\begin{Lemma}\label{qsplitlem1}
Every stable conjugacy class of twisted Levi subgroups
contains a quasi-split member.
\end{Lemma}

\begin{proof}
This follows immediately from \cite[Lemma 6.4]{Kaletha19b} applied to $\xi = \tn{id}$.
\end{proof}

\begin{Lemma} \label{lem:surjectivecenter}
Let $G$ be a quasi-split reductive group
and $T\subseteq G$ a minimal Levi subgroup.
Then the map $H^1(F,Z(G)) \to H^1(F,T)$ is surjective.
\end{Lemma}

\begin{proof}
It suffices to show that $H^1(F,T/Z(G)) = 1$.
This vanishing follows from the fact that the torus $T/Z(G)$ is induced,
that is, a product of Weil restrictions of split tori:
since $T/Z(G)$ is a minimal Levi subgroup of the adjoint quotient of~$G$,
the $\Gamma$-module $X^*(T/Z(G))$ has a $\Gamma$-stable basis,
namely, some $\Gamma$-stable basis of the root system since $G$ is quasi-split.
\end{proof}

We can now interpret stable conjugacy classes of twisted Levi subgroups on the Galois side.

\begin{Lemma} \label{lem:dualizetwistedlevis}
There is a canonical bijection between the sets of
\begin{enumerate}
\item
stable conjugacy classes $J$ of twisted Levi subgroups of~$G$,
\item
$W(G,T_G)$-conjugacy classes $P$ of pairs $(z,U)$ with
$z\in Z^1(F,W)$ and $U\subseteq R(G,T_G)$
such that the action of $\Gamma$ on $R(G,T_G)_z$
preserves some basis of~$U$, and
\item
$\widehat G$-conjugacy classes $\widehat J$ of pairs $(\scrM,s)$
with $\scrM\subseteq\widehat G$ a Levi subgroup and
$s\colon\Gamma\to W({}^LG,\scrM)$
a section of the projection to~$\Gamma$.
\end{enumerate}
\end{Lemma}

\noindent In the second part, the subscript $z$ in $R(G,T_G)_z$
twists the $\Gamma$-set structure on $R(G,T_G)$ using~$z$,
and in the third part, $W({}^LG,\scrM)\defeq N_{{}^LG}(\scrM)/\scrM$.

\begin{proof}
Note that if a section $s$ exists then
the projection $N_{{}^LG}(\scrM)\to\Gamma$ is surjective,
or in other words, the $\widehat G$-conjugacy class
of $\scrM$ is $\Gamma$-stable.
Recall our fixed $F$-pinning $(\scrB_G,\scrT_G,\{X_{\hat\alpha_G}\})$ of~$G$
and $\Gamma$-stable pinning $(B_G,T_G,\{X_{\alpha_G}\})$ of~$\widehat G$.
The identification $X^*(T_G) \simeq X_*(\scrT_G)$ induces
a $\Gamma$-equivariant identification of Weyl groups
$W(G,T_G) \simeq W(\widehat G,\scrT_G)$ and based root data
$(R^\vee(G,T_G),\Delta^\vee_G) \simeq (R(\widehat G,\scrT_G),\Delta_{\widehat G})$,
the bases being obtained from $B_G$ and $\scrB_G$.

For the bijection $(1)\leftrightarrow (2)$,
given $J$, there is a quasi-split element $M$ of $J$ defined over~$F$
by \cite[Lemma 6.4]{Kaletha19b}.
Choose a Levi subtorus $T_M\subseteq M$
and an element $g\in G(\ov{F})$ such that $T_M = {}^g T_G$.
Then $dg\in Z^1(F,N_G(T_G))$.
Let $z$ be the image of $dg$ in $Z^1(F,W(G,T_G))$,
let $U \defeq R({}^{g^{-1}}M,T_G)$,
and let $P$ be the $W$-conjugacy class of $(z,U)$.

Let us verify that the construction is independent of choices.
Choosing $g'$ instead of $g$ conjugates $(z,U)$
by the image in $W(G,T_G)$ of $g'^{-1}g$.
Since all Levi subtori of $M$ are $M(F)$-conjugate,
if we choose $T_M'$ instead of~$T_M$
with $T_M' = {}^m T_M$ for $m\in M(F)$,
then choosing the element $mg$ to conjugate $T_G$ to~$T_M'$
produces the pair $(z,U)$.
Lastly, if $M'$ is another quasi-split element of~$J$ defined over~$F$
with Levi subtorus $T_{M'}$, then there is $h\in G(\ov{F})$
that takes $(M,T_M)$ to $(M',T_{M'})$
and is such that the resulting map between these pairs is defined over~$F$.
If we choose $g$ and $g'$ with $T_M = {}^g T_G$
and $T_{M'} = {}^{g'} T_G$ and such that $h = g' g^{-1}$,
then the two resulting pairs $(z,U)$ and $(z',U')$ will be equal.

All in all, we have constructed a map $(1)\to(2)$.
This map is surjective by Raghunathan's theorem \cite[Corollary~1.3]{Raghunathan04}.
To see that the map is injective, suppose that
two twisted Levi subgroups $M$ and~$M'$
together with the choices $T_M = {}^gT_G\subseteq M$
and $T_{M'} = {}^{g'}T_G\subseteq M'$
as above produce $W$-conjugate pairs $(z,U)$ and $(z',U')$.
After possibly translating $g'$ on the right by an element of $N_G(T_G)(\ov{F})$,
we may assume that $(z,U) = (z',U')$.
Working backwards from the condition $z = z'$,
we find that the map $T_M\to T_{M'}$ of conjugation by $g'g^{-1}$
is defined over~$F$,
or in other words, that $d(g'g^{-1})\in Z^1(F,T_M)$.
It is now clear that $(M,T_M)$ is $G(\ov{F})$-conjugate to $(M',T_{M'})$,
but not yet that the two are stably conjugate.
By Lemma~\ref{lem:surjectivecenter},
there is $h\in G(\ov{F})$ such that $dh\in Z^1(F,Z(M))$
and the image of $dh$ in $Z^1(F,T_M)$
is cohomologous to~$d(g'g^{-1})$,
meaning that $d(g'g^{-1}t) = dh$
for some $t\in T(\ov{F})$.
Replacing $g$ by $t^{-1}g$ in our original choice
then ensures that $g'g^{-1}$
will not only stably conjugate $T_M$ to~$T_{M'}$,
but also $M$ to~$M'$.
Hence $(1)\to(2)$ is a bijection.

For the bijection $(2)\leftrightarrow(3)$,
given $\widehat J$, let $(\scrM,s)\in\widehat J$.
Choose a Borel-torus pair $(\scrB_\scrM,\scrT_\scrM)$ in~$\scrM$
and choose $g\in\widehat G$ taking this pair into $(\scrB_G,\scrT_G)$.
The choice of Borel-torus pair yields an identification
\[
N_{W({}^LG,\scrT)}(R(\scrM,\scrT),\Delta)  \xrightarrow\sim W({}^LG,\scrM)
\]
(see \cite[Corollary~3]{Howlett80}),
where $\Delta$ is the basis corresponding to~$\scrB_\scrM$.
Composing $s$ with the inverse of this identification
and the inclusion into $W({}^LG,\scrT_\scrM)$
results in a homomorphism $\Gamma\to W({}^LG,\scrT_\scrM)$,
and after conjugation by~$g$,
a homomorphism $\Gamma\to W({}^LG,\scrT_{\widehat G})$,
whose first component is the desired cocycle
$z\in Z^1(F,W(\widehat G,\scrT_{\widehat G}))$.
Let $U = R({}^g\scrM,\scrT_{\widehat G})$.
We leave it to the reader to verify that
the $W(\widehat G,\scrT_{\widehat G})$-conjugacy class of $(z,U)$
is independent of the choices of $\scrM$, $\scrT_\scrM$,
$\scrB_\scrM$, and~$g$.
The reverse construction follows the same lines,
and one may even choose $\scrM$ to contain~$\scrT_{\widehat G}$.
\end{proof}

\noindent Once we fix a twisted Levi subgroup $M\subseteq G$ and pair $(\widehat M,s)$,
there is a bijection between
twisted Levi subgroups of~$G$ containing~$M$
and Levi subgroups of~$\widehat G$ containing~$\widehat M$
and stabilized by~$s$.

\begin{Construction}
Let $\phi$ be an $L$-parameter for~$G$.
Whenever $\phi$ normalizes a Levi subgroup $\mc{M}$ of~$G$,
the conjugation action of $W_F$ on~$\mc{M}$ induced by~$\phi$
yields a map $W_F\to W({}^LG,\mc{M})$,
which factors through a finite quotient
and thus yields in turn a map $\Gamma\to W({}^LG,\mc{M})$.
Since $\phi(\SL_2)\subseteq\mc{M}$,
as $N_{\widehat G}(\mc{M})^\circ=\mc{M}$,
the same final map results whether we
obtain the action of $W_F$
by naively restricting $\phi$ to~$W_F$
or passing to its semisimplification.
Together with Lemma~\ref{lem:dualizetwistedlevis},
this construction assigns to every $\phi$-stable
Levi subgroup of~$\widehat G$
a stable conjugacy class of twisted Levi subgroups of~$G$.
\end{Construction}

The data appearing in (3) of Lemma~\ref{lem:dualizetwistedlevis}
can be used to give a coarse definition of Kaletha's $L$-group ${}^LM_\pm$,
a (typically non-split) extension of $\widehat M$ by~$\Gamma$.

\begin{Lemma} \label{lem:lgrouppullback}
Let $M$ be a twisted Levi subgroup of~$G$
with corresponding pair $(\widehat M,s)$ as in Lemma~\ref{lem:dualizetwistedlevis}.
\begin{enumerate}
\item
The group ${}^LM_\pm$ is the pullback of the extension
$N_{{}^LG}(\widehat M)$ of $\widehat M$ by $W({}^LG,\widehat M)$ along~$s$:
\begin{center}
\begin{tikzcd}
1 \rar & \widehat M \rar\dar[equals] & {}^LM_\pm \rar\dar & \Gamma \rar\dar{s} & 1 \\
1 \rar & \widehat M \rar & N_{{}^LG}(\widehat M) \rar & W({}^LG,\widehat M) \rar & 1.
\end{tikzcd}
\end{center}

\item
An $L$-parameter~$\phi$ for~$G$ factors through the embedding
${}^LM_\pm\to {}^LG$ if and only if $\phi$ normalizes $\widehat M$
and the map $\Gamma\to W({}^LG,\widehat M)$
induced by~$\phi$ is $W(\widehat G,\widehat M)$-conjugate to~$s$.
\end{enumerate}
\end{Lemma}

\begin{proof}
For the first part, the commutativity of this diagram follows from 
the definition of Kaletha's $L$-embedding
${}^LM_\pm\to{}^LG$, \cite[(6.1)]{Kaletha19b}:
his element $\sigma_{M,G}\rtimes\sigma_G$
maps in $W({}^LG,\widehat M)$ to our~$s(\sigma)$.
The fact that the righthand square is a pullback then follows,
by a diagram chase, from the exactness of the two rows.
For the second part, use the universal property of the pullback.
\end{proof}

Despite the usefulness of Lemma~\ref{lem:lgrouppullback},
it is very far from specifying an $L$-embedding ${}^LM_\pm\to{}^LG$,
even up to conjugacy.
Kaletha's work goes much further and constructs a conjugacy class
of $L$-embeddings extending the canonical conjugacy class of embeddings
$\widehat M\to\widehat G$,
using the Tits lift and the Langlands--Shelstad theory of gauges.
Because the details of the construction are not critical for our application,
we refer the reader to \cite[\S~6.1]{Kaletha19b} for more details.

\begin{Proposition}[{\cite[Lemmas 6.11 and 6.12]{Kaletha19b}}]\label{tLeviembprop}
For $T \subseteq M \subseteq G$ a twisted Levi subgroup with maximal torus $T$:
\begin{enumerate}
\item
The canonical $\widehat{G}$-conjugacy of embeddings $\widehat{M} \hookrightarrow \widehat{G}$ extends to a $\widehat{G}$-conjugacy class of $L$-embeddings $^{L}M_{\pm} \to \prescript{L}{}G$.
\item
The canonical $\widehat{M}$-conjugacy class of embeddings $\widehat{T} \to \widehat{M}$ extends to a $\widehat{M}$-conjugacy class of $L$-embeddings $^{L}T_{\pm} \to \prescript{L}{}M_{\pm}$.
\item
Every morphism ${}^LT_\pm\to{}^LG$ in this conjugacy class factors
as the composition of morphisms in the conjugacy classes of the two previous points.
\end{enumerate}
\end{Proposition}

\begin{proof}
If $t_{M,G}$ and $t_{T,M}$ denote the Tits cocycles corresponding to the maps $\sigma \mapsto n(\sigma_{(M,G)})$ and $\sigma \mapsto n(\sigma_{(T,M)})$ respectively, then the claimed maps are given by 
\begin{equation*}
\widehat{M} \boxtimes_{t_{M,G}} \Gamma \to \widehat{G} \rtimes W_{F}, \hspace{.25cm} \widehat{\jmath}_{M,G}(m) \boxtimes \sigma \mapsto m\cdot n(\sigma_{M,G}) \rtimes \sigma 
\end{equation*}
and 
\begin{equation*}
\widehat{T} \boxtimes_{t_{T,M}} \Gamma \to\widehat{M} \boxtimes_{t_{M,G}} \Gamma, \hspace{.25cm} \widehat{\jmath}_{T,M}(t) \boxtimes \sigma \mapsto t\cdot n(\sigma_{T,M}) \boxtimes \sigma.
\end{equation*}
The compatibility of these with $^{L}T_{\pm} \to \prescript{L}{}G$ follows from the identity $n(\sigma_{T,M}) \cdot n(\sigma_{M,G}) = n(\sigma_{T,G})$. We refer the reader to \cite{Kaletha19b} for the details and the proof of the first statement.
\end{proof}

To finish, we discuss Weyl group actions.
Recall that, for a quasi-split connected reductive group $M$,
there is a short exact sequence of $F$-group schemes 
$ 1 \to \underline{\tn{Inn}}_{M} \to \underline{\tn{Aut}}_{M}
\to \underline{\tn{Out}}_{M} \to 1 $
which is split via a homomorphism corresponding to a choice of $\Gamma$-stable pinning $\mc{P}_{M}$ for $M$. This splitting gives an identification
$\underline{\tn{Out}}_M(F)\isoto\tn{Aut}_\Gamma(\widehat M)/\tn{Inn}_\Gamma(\widehat M)$,
as explained in \cite[\S~2.3.4]{Kaletha23}, and also an inclusion
$s\colon\underline{\tn{Out}}_{M}(F)
\hookrightarrow \tn{Aut}_{\Gamma}(\widehat{M})$,
where every automorphism in the image of $s$ preserves the fixed pinning
$\mc{P}_{\widehat{M}}$. When $M$ is in addition a twisted Levi subgroup of~$G$,
we can compose $s$ with the map $W(G,M)(F) \to \underline{\tn{Out}}_{M}(F)$
to obtain an action of the group $W(G,M)(F)$
by $\Gamma$-stable automorphisms of $\widehat{M}$
which preserve the pinning $\mc{P}_{\widehat{M}}$.
Denote the $F$-rational automorphism of $M$
(resp.\ $\Gamma$-equivariant automorphism of $\widehat{M}$)
corresponding to $w \in W(G,M)(F)$ by $\theta_{w}$ (resp.\ $\theta_{w}^{\vee}$).
It is straightforward to check that the conjugacy classes
of admissible embeddings $M\into G$ and $\widehat M\into \widehat G$
are stable under the action of $W(G,M)(F)$.

\subsection{Universal covers} \label{sec:univ}
Given a complex torus~$S$, define $S_\tn{univ} \defeq \varprojlim_n S$,
the inverse limit over $n$th power maps.
In other words, $S_\tn{univ}$
is the complex group scheme of multiplicative type such that
$X^*(S_\tn{univ}) = X^*(S)\otimes\bbQ$.
More generally, given a complex reductive group~$H$, define
the (algebraic) \textit{universal cover} of~$H$ by
\[
H_\tn{univ} \defeq Z(H)^\circ_\tn{univ}\times H_\tn{sc},
\]
where $Z(H)^\circ_\tn{univ} \defeq (Z(H)^\circ)_\tn{univ}$.
The universal cover is functorial in the following sense.

\begin{Lemma} \label{lemunivfunctorial}
Every isomorphism $i\colon H \simeq H'$ of complex reductive groups
admits a unique lift to an isomorphism
$i_\tn{univ}\colon H_\tn{univ} \simeq H_\tn{univ}'$ of universal covers.
\end{Lemma}

\begin{proof}
Clearly there exists such a lift $i_\tn{univ}$:
when $H$ is a torus this follows from the description
of the universal cover in terms of character lattices,
when $H$ is semisimple this is a well-known property of the simply-connected cover,
and in general one takes the product of the lifts from these two cases.
For uniqueness, we need only show that any lift $i_\tn{univ}$
is the product of lifts $Z(H)^\circ_\tn{univ} \to Z(H')^\circ_\tn{univ}$
and $H_\tn{sc} \to H'_\tn{sc}$.
There are no nonzero homomorphisms $H_\tn{sc} \to Z(H')^\circ_\tn{univ}$,
and any homomorphism $Z(H)^\circ_\tn{univ}\to H'_\tn{sc}$ fitting into a lift
would necessarily factor through the preimage of $Z(H')$ in $H'_\tn{sc}$,
namely, $Z(H_\tn{sc}')$.
However, there are no nonzero maps from the universal cover of a torus
to a finite multiplicative group.
\end{proof}

\begin{Notation} \label{plusnotation}
Given a subgroup $A$ of~$H$, we write $A^+_{(H)}$
for the preimage of $A$ in~$H_\tn{univ}$.
\end{Notation}

We warn the reader that this notation is different from the one used in \cite{Kaletha16a}.
In that setting, where $A$ is a $\Gamma$-stable subgroup of~$\widehat G$,
it is used to denote the preimage of $A^\Gamma$ in $\widehat G_\tn{univ}$.
In our notation, this would be denoted by $A^{\Gamma,+}$.

Using the description of isogenies of reductive groups in terms of root data
(see, for instance, \cite[II.1.13]{Jantzen03}),
one can show that if $S$ is a maximal torus of~$H$, then
\begin{equation} \label{dualunivtorus}
X^*(S^+_{(H)}) = (\bbZ R^\vee(H,S))^*
\defeq \{v\in X^*(S)\otimes\bbQ \colon \langle v,\alpha^\vee\rangle\in\bbZ
\tn{ for all }\alpha\in R(H,S)\},
\end{equation}
where $\bbZ R^\vee(H,S)$ is the $\bbZ$-linear span of the set~$R^\vee(H,S)$
and where the dual is taken in the $\bbQ$-vector space $X^*(S)\otimes\bbQ$.
This dual will be a lattice if and only if $H$ is semisimple.

\begin{Lemma} \label{toruscoversquares}
Let $M_1,M_2\subseteq H$ be Levi subgroups
containing a maximal torus~$S$
and let $A\subseteq S$ be a subgroup.
The square at left is a pullback
and the square at right is a pushout:
\[
\begin{tikzcd}[every arrow/.append style={two heads}]
A^+_{(M_1\cap M_2)} \rar\dar &
A^+_{(M_1)} \dar \\
A^+_{(M_2)} \rar & A^+_{(H)}.
\end{tikzcd}
\qquad
\begin{tikzcd}[every arrow/.append style={hook}]
X^*(\pi_0(A^+_{(H)})) \rar\dar &
X^*(\pi_0(A^+_{(M_1)})) \dar \\
X^*(\pi_0(A^+_{(M_2)})) \rar &
X^*(\pi_0(A^+_{(M_1\cap M_2)})).
\end{tikzcd}
\]
\end{Lemma}

\begin{proof}
For the left square, it suffices to prove the result for $S=A$,
since that square is the pullback of the analogous square for~$S$
along the maps induced by the inclusion $A\into S$.
Passing to character groups, it suffices to prove that
the image of the square under the functor $X^*$ is a pushout,
and this claim amounts, by \eqref{dualunivtorus},
to showing that
\[
(\bbZ R^\vee(M_1\cap M_2,S))^* =
(\bbZ R^\vee(M_1,S))^* + (\bbZ R^\vee(M_2,S))^*.
\]
The equality above follows from the general fact
that if $A$ and $B$ are free subgroups of a rational vector space~$V$
then $A^* + B^* = (A\cap B)^*$ whenever $A/(A\cap B)$ and $B/(A\cap B)$
are free, so that we may use bases to compute the duals.
This general fact applies to our situation because
any root system basis of $R^\vee(H,S)$
contains bases of the Levi subsystems $R^\vee(M_1,S)$ and $R^\vee(M_2,S)$.

For the right square, we have already seen in the previous paragraph
that the square would be a pushout if we removed the $\pi_0$'s from the diagram.
Since $X^*(\pi_0(A))$ is the torsion subgroup of $X^*(A)$,
for any multiplicative type group~$A$,
it suffices to show that if we apply the torsion subgroup functor
to any commutative square of abelian groups having the properties that
(1) all maps are injections and (2) the square is a pushout,
then the image square again has these two properties.
Checking this compatibility is a straightforward exercise.
\end{proof}

\subsection{Tate--Nakayama duality and restriction to the band}%
In this subsection, we give interpretations in terms of $\widehat G$
of various aspects of the Newton and Kottwitz map.

We have the following dual interpretation of the diagram
\eqref{TNbanddiag} from Section~\ref{admsbgpsubsec}:

\begin{Proposition}\label{Newtonbandprop}
For each twisted Levi subgroup $M$ of~$G$,
there is a canonical injection 
\begin{equation}\label{Newtonbandmap}
X^{*}(\pi_{0}(Z(\widehat{G})^{\Gamma,\circ,+}_{(\widehat M)}))
\hookrightarrow \Hom_{\ov{F}}(\mu, Z(M))
\end{equation}
making the diagram below commute,
where the top map is restriction,
the bottom map is restriction to $u$ followed by the isomorphism of Lemma~\ref{bandhoms},
and the left map is Tate--Nakayama duality:
\begin{equation}\label{dualTNdiag}
\begin{tikzcd}
X^{*}(\pi_{0}(Z(\widehat{M})^{\Gamma,+}))
\arrow{r} \arrow["\iota^{-1}","\sim"']{d}& X^{*}(\pi_{0}(Z(\widehat{G})^{\Gamma,\circ,+}_{(\widehat M)})) \arrow[hookrightarrow]{d} \\
H^{1}(\cE, Z(M) \to M) \arrow{r} & \Hom_{\ov{F}}(\mu, Z(M)).
\end{tikzcd}
\end{equation}
\end{Proposition}

\begin{proof}
First, fix some $n\in\mathbb{N}$ and let
$Z_n$ be the preimage of $Z(M_\tn{der})$ in $Z(M)$
under the $n$th-power map,
so that $\widehat M_\tn{univ} \simeq \varprojlim \widehat{M/Z_n}$.
We will prove the analogous result with $Z(\widehat G)^{\Gamma,\circ,+}$
replaced by the preimage of $Z(\widehat G)^{\Gamma,\circ}$ in $\widehat{M/Z_n}$,
with $Z(M)$ replaced by $Z_n$,
and where the ``$+$'' superscript is used for preimages of subgroups of
$Z(\widehat{M})$ in $\widehat{M/Z_n}$,
suppressing the subscript ``$(\widehat{M/Z_n})$'' for readability.

We identify $\widehat{M/Z_n} \to \widehat{M}$ with the map
$\widehat{M}_{\Sc} \times Z(\widehat{M})^{\circ} \to \widehat{M}$
given by the usual map on the first factor and the $n$th-power map on the second,
writing $\widehat{Z_n}$ for the kernel of the map.
Under this identification, $Z(\widehat{G})^{\Gamma,\circ,+}$
is the subgroup of elements in $Z(\widehat{M})^\circ$
whose $n$th power lies in $Z(\widehat{G})^{\Gamma,\circ}$.
Then $(Z(\widehat{G})^{\Gamma,+,\circ})^\circ = Z(\widehat{G})^{\Gamma,\circ}$
and the inclusion
$\widehat{Z_n} \hookrightarrow Z(\widehat{G})^{\Gamma,\circ,+}$ 
induces an isomorphism
\begin{equation*}
\frac{\widehat{Z_n}}{\{\tn{id}\} \times (Z(\widehat{G})^{\Gamma,\circ})[n]}
\longisoto \pi_{0}(Z(\widehat{G})^{\Gamma,\circ,+}).
\end{equation*}
The desired map (at level $n$) is the composition 
\begin{equation}\label{klevelmap1}
X^{*}( \pi_{0}(Z(\widehat{G})^{\Gamma,\circ,+})) \isoto X^{*}(\widehat{Z_n}/(Z(\widehat{G})^{\Gamma,\circ})[n]) \hookrightarrow X^{*}(\widehat{Z_n}) \isoto \Hom_{\ov{F}}(\mu_n, Z_n),
\end{equation}
where the rightmost map is given by the identifications
\begin{equation*}
X^{*}(\widehat{Z_n}) \isoto \Hom_{\Z}(X^{*}(Z_n), \Q/\Z)
= \Hom_{\ov{F}}(\mu_n, Z_n).
\end{equation*}
It is clear that the map \eqref{klevelmap1}
is compatible with the projection maps as $n$ varies,
giving a well-defined injection
$X^{*}(\pi_{0}(Z(\widehat{G})^{\Gamma,\circ,+}))
\hookrightarrow \Hom(\mu_{\ov{F}}, Z(M)_{\ov{F}})$ as claimed.

To check compatibility with the Tate--Nakayama isomorphism,
we may again work at the $n$th level.
Choose an elliptic maximal torus $T$ of~$M$.
By \eqref{TNbanddiag} and the ensuing discussion,
following the diagram \eqref{dualTNdiag} down and then to the right
yields the map
\begin{equation*}
X^{*}(\pi_{0}(Z(\widehat{M})^{\Gamma,+}))
\longisoto \frac{(X_{*}(T/Z_n)/X_{*}(T_{M,\psc}))^{N}}{I\cdot (X_{*}(T)/X_{*}(T_{M,\psc}))}
\longrightarrow \frac{X_{*}(T/Z_n)}{X_{*}(T)}
\end{equation*}
and following the same diagram in the other direction corresponds to the composition
\begin{equation*}
X^{*}(\pi_{0}(Z(\widehat{M})^{\Gamma,+}))
\to X^{*}(\widehat{Z_n}/(Z(\widehat{G})^{\Gamma,\circ}[n])) \to X^{*}(\widehat{Z_n})
\isoto \frac{X_{*}(T/Z_n)}{X_{*}(T)},
\end{equation*} 
where the first map is restriction,
using that $\widehat{Z_n} \cap (Z(\widehat{M})^{\Gamma,+,\circ})
= (Z(\widehat{G})^{\Gamma,\circ})[n]$,
and the last map is the canonical identification used in \eqref{klevelmap1}.
The two preceding displayed composite maps agree, giving the result.
\end{proof}

\begin{Corollary}\label{isogenyfaccor}
Let $M \into N \into G$ be a chain of twisted Levi subgroups
and let $\widehat M\into\widehat N\into\widehat G$
be dual admissible embeddings.
The corresponding commutative diagram below is a pullback square:
\[
\begin{tikzcd}
X^*\bigl(\pi_0\bigl(Z(\widehat{G})^{\Gamma,\circ,+}_{(\widehat N)}\bigr)\bigr)
\rar[hookrightarrow] \dar[hookrightarrow] &
X^*\bigl(\pi_0\bigl(Z(\widehat{G})^{\Gamma,\circ,+}_{(\widehat M)}\bigr)\bigr)
\dar[hookrightarrow] \\
\Hom_{\ov{F}}(\mu,Z(N)) \rar[hookrightarrow] &
\Hom_{\ov{F}}(\mu,Z(M)).
\end{tikzcd}
\]
\end{Corollary}

\begin{proof}
The image of $\chi$ lies in $\Hom_{\ov{F}}(\mu, Z(M))$ if and only if, by the commutativity of \eqref{dualTNdiag}, it has a preimage $\tilde{\chi} \in X^{*}(\pi_{0}(Z(\widehat{M})^{\Gamma,+}))$ whose image $[x] \in H^{1}(\cE, Z(M) \to M)$ via Tate--Nakayama lies in the subset $H^{1}(\cE, Z(N) \to M)$. We obtain the desired result by applying functoriality of the Tate--Nakayama duality isomorphism.
\end{proof}

\begin{Definition} \label{gphiregular}
Let $\phi$ be an $L$-parameter for $G$
and let $\mc{N}$ be a Levi subgroup of~$G$ normalized by~$\phi$.
A character $\chi\in X^*(\pi_{0}(Z(\mc{N})^{\Gamma,+}_{(\mc{N})}))$
is \textit{$(G,\phi)$-regular} if it does not lie in $X^*(\pi_{0}(Z(\mc{N})^{\Gamma,+}_{(\mc{M})}))$ for any Levi subgroup $\mc{N}\subsetneq\mc{M}\subseteq\widehat G$. Write $X^*(\pi_{0}(Z(\mc{N})^{\Gamma,+}_{(\mc{N})}))_\Gphireg$ for the set of these.
\end{Definition}

\begin{Remark}\label{restrem}
By Proposition \ref{Newtonbandprop}, the $G$-regularity of $\chi\in X^*(\pi_{0}(Z(\mc{N})^{\Gamma,+}_{(\mc{N})}))$ is determined by its restriction to $\pi_{0}(Z(\widehat{G})^{\Gamma,\circ,+}_{(\widehat{N})})$.
\end{Remark}

Recall \eqref{infB}, which defines
$\wt{Y}_{+,\tn{tor}}(B\to G)$ for $B$ infinite.

\begin{Proposition} \label{TNgphireg}
Let $\phi$ be an $L$-parameter of~$G$,
let $\mc{N}$ be a Levi subgroup of~$G$ normalized by~$\phi$,
and let $N$ be a corresponding quasi-split twisted Levi subgroup of~$G$.
Then the Tate--Nakayama duality isomorphisms restrict to isomorphisms
\[
\begin{tikzcd}%
H^1_\tn{bas}(\cE,N)_\Greg \rar{\sim} %
& \ov{Y}_{+,\tn{tor}}(Z(N)\to N)_\Greg \rar{\sim}%
& X^*(\pi_0(Z(\widehat N)^{\Gamma,+}_{(\widehat N)}))_\Gphireg %
\end{tikzcd}
\]
\end{Proposition}

\begin{proof}
The first isomorphism follows from Lemma~\ref{kottnewtcompat},
and the second from Corollary~\ref{isogenyfaccor}.
\end{proof}

To finish, we generalize Tate--Nakayama duality to the non-basic setting.

\begin{Proposition}\label{dualKott1} %
Let $N$ be a quasi-split twisted Levi subgroup of~$G$.
\begin{enumerate}
\item
There is a canonical identification
$\wt{Y}_{+,\tn{tor}}(Z(N)\to G)
\longisoto X^{*}(\pi_{0}(Z(\widehat{G})^{\Gamma,+}_{(\widehat N)}))/W(G,N)(F)$.
\item
This identification is functorial for embeddings
$G\into G'$ of $G$ as a twisted Levi subgroup.
\end{enumerate}
\end{Proposition}

\noindent Taking $W(G,N)(F)$-orbits makes this map independent of the choice of embedding.

\begin{proof}
This essentially follows from the proof of \cite[Proposition 5.3]{Kaletha16a},
which we summarize.
As explained there, for any finite central subgroup $A$ in $N$ there is an isomorphism
\begin{equation*}
\wt{Y}_{+,\tn{tor}}(A\to G;T)
\longisoto X^{*}(\pi_{0}(Z(\widehat{G})_{(\widehat{T/A})}^{\Gamma,+})),
\end{equation*}
Taking $W(G,N)(F)$-orbits and then the colimit over all such
$A$ gives the identification,
and functoriality is straightforward to check.
\end{proof}

Using Proposition~\ref{dualKott1},
we can construct a unique map
\begin{equation} \label{fullTN}
\begin{tikzcd}
\tau\colon H^1_\tn{L-reg}(\cE,G) \rar[hookrightarrow] &
\displaystyle\bigsqcup_{[N]}
\frac{X^{*}(\pi_{0}(Z(\widehat{G})^{\Gamma,+}_{(\widehat N)}))}{W(G,N)(F)},
\end{tikzcd}
\end{equation}
where $[N]$ ranges over stable conjugacy classes of twisted Levi subgroups of~$G$,
whose restriction to each set $H^1_\tn{reg}(\cE,G)_N$ is the composition
\[
\begin{tikzcd}
H^1_\tn{reg}(\cE,G)_N \rar[hookrightarrow]{\kappa} &
\ov{Y}_{+,\tn{tor}}(G)_N \rar[hookrightarrow] &
\wt{Y}_{+,\tn{tor}}(Z(N)\to G) \rar{\sim} &
\dfrac{X^{*}(\pi_{0}(Z(\widehat{G})^{\Gamma,+}_{(\widehat N)}))}{W(G,N)(F)},
\end{tikzcd}
\]

\section{The rigid local Langlands conjectures for double covers}\label{llcsec}

\subsection{The basic conjectures}\label{basicdoublesubsec}
First we recall the rigid refined local Langlands conjectures for quasi-split $G$.
Let $\phi$ be an $L$-parameter for~$G$ and let $\mathfrak{w}$ be a Whittaker datum for $G$. 
Kaletha's basic rigid local Langlands conjectures \cite[(5.7)]{Kaletha16a}
can be reformulated uniformly, removing the dependence on~$Z$, as follows.

\begin{Conjecture}\label{rigidLLC1} 
There is a commutative diagram with horizontal bijections
\[
\begin{tikzcd}[row sep=small]
\Pi_\phi^{\cE,\tn{bas}} \arrow{r}{\iota_{\mathfrak{w}}}[swap]{\sim} \arrow{d} & \Irr(\pi_{0}(S_{\phi}^{+})) \arrow{d} \\
H^1_\tn{bas}(\mc{E}, G) \rar{\sim} & \pi_{0}(Z(\widehat{G})^{\Gamma,+})^{*}.
\end{tikzcd}
\]
\end{Conjecture}

Here $\Pi_\phi^{\cE,\tn{bas}}$ is a set of isomorphism classes of representations
$(G',\psi',z',\pi)$ of rigid inner twists of $G$ as in \cite[\S~5.1]{Kaletha16a}
but where no finite central subgroup $Z$ is fixed,
the $+$-notation forms preimages in $\widehat{G}_\tn{univ}$
as in Notation~\ref{plusnotation},
the bottom map is Tate--Nakayama duality isomorphism of Theorem~\ref{Tatenak1},
the left-hand column extracts the underlying cohomology class,
and the right-hand column is formation of central characters.
The bijection $\iota_\mathfrak{w}$ is expected to satisfy many additional properties,
such as the endoscopic character identities (cf. \cite[\S~5.4]{Kaletha16a}).

Our goal in this subsection is to extend Conjecture \ref{rigidLLC1}
to Kaletha's double covers $M(F)_\pm$ of twisted Levi subgroups of~$G$,
and to prove that the two variants of the conjecture are equivalent.

On the automorphic side, we will not recall the precise definition of $M(F)_\pm$,
referring the reader to \cite[\S~3.1, Definitions~6.2 and~6.10]{Kaletha19b}.
We note only that $M(F)_\pm$ is equipped with a map $M(F)_\pm\to M(F)$
whose kernel $\{\pm1\}$ is normal of order two,
and that the construction of $M(F)_\pm$ depends implicitly on the embedding of~$M$ in~$G$,
though this dependence is not reflected in the notation.
Call a character $\theta\colon M(F)_{\pm} \to \mathbb{C}^{\times}$ \textit{genuine}
if $\theta(-1) = -1$.

On the Galois side, an \textit{$L$-parameter} for a double cover $M(F)_{\pm}$
is a continuous homomorphism $W_{F}' \to {}^LM_{\pm}$
that commutes with the two maps to $\Gamma$
and has semisimple restriction to~$W_F$. 
Two such parameters are \textit{equivalent} if they are $\widehat{M}$-conjugate. 
A first indication of the utility of this notion
is the local Langlands correspondence for $S(F)_\pm$.

\begin{Theorem}[{\cite[Theorem~3.16.1]{Kaletha19b}}]
There is a canonical bijection between the set of genuine characters of $S(F)_{\pm}$ and equivalence classes of $L$-parameters $W_{F} \to \prescript{L}{}S_{\pm}$.
\end{Theorem}

Next, we extend the notion of rigid inner twists to double covers.
Given an inner twist $\psi\colon M_{\ov{F}} \to M'_{\ov{F}}$, the induced map on abelianizations is defined over~$F$. We can use this map to transfer the data $R(M_{\tn{ab}}, G) \to X^{*}(M_{\tn{ab}})$ defining the double cover $M(F)_\pm$ (see \cite[\S~6.3]{Kaletha19b}) to similar data for $M'$, via the composition 
\begin{equation*}
R(M_{\tn{ab}}, G) \to X^{*}(M_{\tn{ab}}) \xrightarrow{\psi^{-1}} X^{*}(M'_{\tn{ab}}).
\end{equation*}
We then define $M'_{\tn{ab}}(F)_{\pm}$ using this data, and define $M'(F)_{\pm}$ as the pullback of the diagram $M'(F) \to M'_{\tn{ab}} \leftarrow M'_{\tn{ab}}(F)_{\pm}$.
Recall that an isomorphism $(M_1,\psi_1,z_1)\isoto(M_2,\psi_2,z_2)$
of rigid inner twists of~$M$
is a pair $(f,m)$  where $f\colon M_{1} \to M_{2}$ is an $F$-rational isomorphism,
$m \in M(\ov{F})$ satisfies $\psi_{1} \circ \Inn(m) = f \circ \psi_{2}$,
and ${}^m z_1 = z_2$.

\begin{Definition}\label{pmRITdef} \begin{enumerate}
\item
A \textit{rigid inner twist} of $M(F)_{\pm}$ is a triple $(M'(F)_{\pm}, \psi, z)$
with $(M', \psi, z)$ a rigid inner twist of $M$
and $M'(F)_{\pm}$ obtained as above from $M(F)_\pm$ and~$\psi$.

\item
An \textit{isomorphism} $(M_{1}(F)_{\pm}, \psi_1, z_{1})
\isoto (M_{2}(F)_{\pm}, \psi_2, z_{2})$
is an isomorphism $(f,m)\colon (M_{1}, z_{1}) \longrightarrow (M_{2}, z_{2})$
of rigid inner twists.
\end{enumerate}
\end{Definition}

\begin{Lemma}\label{fpm}
An isomorphism
$(f,m)\colon (M_{1}(F)_{\pm}, z_{1}) \isoto (M_{2}(F)_{\pm}, z_{2})$
induces an isomorphism of extensions 
$f_\pm\colon M_1(F)_\pm\isoto M_2(F)_\pm$:
\begin{equation}
\begin{tikzcd}\label{doubleRITeq1}
0 \arrow{r} &
\{\pm 1\} \arrow[equals]{d} \arrow{r} &
M_{1}(F)_{\pm} \dar{f_{\pm}}[swap]{\sim} \arrow{r} &
M_{1}(F) \dar{f}[swap]{\sim} \arrow{r} &
1 \\
0 \arrow{r} & \{\pm 1\} \arrow{r} & M_{2}(F)_{\pm} \arrow{r} & M_{2}(F) \arrow{r} & 1.
\end{tikzcd}
\end{equation}
\end{Lemma}

\begin{proof}
By construction, the map $f^{-1}\colon X^{*}(M_{1,\tn{ab}}) \isoto X^{*}(M_{2,\tn{ab}})$ induced by $f^{-1}$ is compatible with the maps of both groups to $X^{*}(M_{\tn{ab}})$.
Hence, by \cite[\S~5]{Kaletha19b}, the isomorphism $M_{1}(F) \isoto M_{2}(F)$ lifts canonically to an isomorphism $f_\pm\colon M_{1}(F)_{\pm} \isoto M_{2}(F)_{\pm}$ making \eqref{doubleRITeq1} commute.
\end{proof}

\begin{Definition}
Recall that a representation $\pi$ of $M(F)_\pm$ is \textit{genuine} if $\pi(-1) = -\tn{id}$.
\begin{enumerate}
\item
A \textit{genuine representation} of the rigid inner twist
$(M'(F)_{\pm},\psi,z)$ is a $4$-tuple $(M'(F)_{\pm},\psi,z,\pi)$
where $\pi$ is a genuine representation of $M'(F)_{\pm}$.
\item
An \textit{isomorphism} $(f,m)\colon (M_{1}(F)_{\pm}, \psi_1, z_{1},\pi_{1})
\isoto (M_{2}(F)_{\pm}, \psi_2, z_{2}, \pi_{2})$
between two such representations is an isomorphism of rigid inner twists
such that the induced isomorphism $f_{\pm}\colon M_1(F)_\pm \isoto M_2(F)_\pm$
identifies the genuine representations $\pi_{1}$ and $\pi_{2}$. 
\end{enumerate}
 \end{Definition}

We can now reformulate Conjecture \ref{rigidLLC1} for double covers.
Let $M$ be a twisted Levi subgroup of~$G$
which we take to be quasi-split by Lemma~\ref{qsplitlem1}
and fix a Whittaker datum $\mathfrak{w}_{M}$ for $M$.
Given an $L$-parameter $\phi_\pm \colon W_{F}' \to {}^LM_{\pm}$,
let $S_{\phi_{\pm}} \defeq Z_{\widehat{M}}(\phi_{\pm})$
and $S_{\phi_\pm}^+ \defeq (S_{\phi_\pm})^+_{(\widehat M)}$.

\begin{Theorem}\label{rigidLLC1pm}
Let $\phi_{M,\pm} \colon W_{F}' \to \prescript{L}{}M_{\pm}$
be an $L$-parameter.
If Conjecture~\ref{rigidLLC1} holds for $M$
and is compatible with twists by cocentral characters,
then there is a finite subset $\Pi_{\phi_{M,\pm}}$
of genuine representations of rigid inner twists of $M(F)_{\pm}$
and a commutative diagram with horizontal bijections
\begin{equation*}
\begin{tikzcd}
\Pi_{\phi_{M,\pm}}^{\cE,\tn{bas}} \arrow["\iota_{\mathfrak{w}_M,\pm}"]{r} \arrow{d} & \Irr(\pi_{0}(S_{\phi_{M,\pm}}^{+})) \arrow{d} \\
H^1_\tn{bas}(\mc{E}, M) \arrow{r} & X^{*}(\pi_{0}(Z(\widehat M)^{\Gamma,+})),
\end{tikzcd}
\end{equation*}
where the bottom map is the one from Theorem \ref{Tatenak1}, the left-hand column extracts the underlying torsor, the right-hand column is induced by taking central characters, and the $+$ superscript forms the preimage in $\widehat M_\tn{univ}$.
\end{Theorem}

\begin{proof}
This proof will mostly follow the arguments of \cite[Remark 6.14]{Kaletha19b}.
We thus fix a finite central subgroup $Z$
and work with $H^1(\cE,Z\to M)$ and preimages in $\widehat{M/Z}$.
Choose $\chi$-data $\{\chi_{\alpha}\}$ for $R(M_{\tn{ab}},G)$. On the one hand, these data determine a genuine character $\chi_{M}$ of $M_{\tn{ab}}(F)_{\pm}$, which then pulls back to a genuine character of $M(F)_{\pm}$. On the other hand, by \cite[Fact 6.13]{Kaletha19b}, these data determine an isomorphism of Weil forms of $L$-groups $\iota_{\chi_M} \colon {}^{L}M \to \prescript{L}{}M_{\pm}$ . Together, $\phi_{M,\pm}$ and $\iota_{\chi_{M}}$ determine an $L$-parameter $\phi\colon W_{F}' \to \prescript{L}{}M$ and an isomorphism $s_\chi\colon S_{\phi_{M}}^{+} \to S_{\phi_{M,\pm}}^{+}$.

Given a representation $(M', \psi, z, \pi) \in \Pi^{Z}_{\phi_{M}}$, the isomorphism $\psi$ transfers the $\chi$-data $\{\chi_{\alpha}\}$ to $\chi$-data for the $\Sigma$-set $R(M'_{\tn{ab}}, G) \to X^{*}(M'_{\tn{ab}})$, which determine a genuine character $\chi_{M'}$ of $M'_{\tn{ab}}(F)_{\pm}$. Twisting $\pi$ by $\chi_{M'}$ gives a genuine representation of $M'(F)_{\pm}$, and we thus obtain an element $(M', \psi, z, \pi \cdot \chi_{M'}) \in \Pi_{\phi_{M,\pm}}^{Z}$. One checks easily that an isomorphic representation of $(M', \psi, z, \pi)$ would yield an isomorphic genuine representation of a rigid inner twist.

Twisting by each $\chi_{M'}$ induces a bijection $\chi_{-}^{-1}\colon \Pi_{\phi_{M}}^{Z} \isoto \Pi_{\phi_{M,\pm}}^{Z}$ and thus we can define $\iota_{\mathfrak{w}_{M},\pm}$ as the composition $s_{\chi} \circ \iota_{\mathfrak{w}} \circ (\chi_{-}^{-1})$. The assumption of compatibility with twisting implies that this composition does not depend on the choice of $\chi$-data. Commutativity of the diagram follows from the analogous commutativity in Conjecture \ref{rigidLLC1}.
\end{proof}

The action of $W(G,N)(F)$ on $N(F)$ and its $L$-parameters
described in Section~\ref{basicdoublesubsec}
extends to an action on $N(F)_\pm$ and its $L$-parameters,
and the expected compatibility with the local Langlands correspondence,
\cite[\S~2.3.4]{Kaletha23}, extends to the setting of double covers.
Given $w\in W(G,N)(F)$ and
a representation $\dot{\pi}_{\pm} = (N', \psi, z,\pi_{\pm})$
of a rigid inner twist of $N(F)_{\pm}$ as in Definition~\ref{pmRITdef}, let
\[
\theta_{w} \cdot (N', \psi, z,\pi_{\pm}) = (N', \psi\circ\theta_w^{-1},
\theta_{w}(z), \pi_{\pm}). 
\]
The group $N'$ is unchanged by Lemma~\ref{miracleLem},
although its rigidification $z$ need not be.

\begin{Conjecture}\label{ICMconj}(\cite[Conjecture 2.12]{Kaletha23})
We have
\[
\Pi^{\cE,\tn{bas}}_{\theta_{w} \circ \phi_{\pm}}
\defeq \theta_{w} \cdot \Pi_{\phi_{\pm}}
= \{\theta_{w} \cdot \dot{\pi} \mid \dot{\pi} \in \Pi_{\phi_{\pm}}\}, 
\qquad
\iota_{\mathfrak{w}_{N},\pm}(\theta_{w} \cdot \dot{\pi})
= (\theta_{w}^{\vee} \circ \phi, \rho \circ (\theta_{w}^{\vee})^{-1}),
\]
where $(\phi, \rho) = \iota_{\mathfrak{w}_{N},\pm}(\dot{\pi})$
is the enhanced rigid parameter corresponding to $\dot{\pi}$.
\end{Conjecture}

\subsection{Factorization of discrete $L$-parameters}
Let $\phi$ be an $L$-parameter for~$G$.
In this subsection we study the set $\tn{Lev}(\widehat G)^\phi$
of Levi subgroups of~$\widehat G$ normalized by~$\phi$,
or equivalently, by Lemma~\ref{lem:lgrouppullback},
the set of $L$-embeddings ${}^LM_\pm\to{}^LG$ through which $G$ factors.

\begin{Lemma} \label{lemfinlevis}
Let $\phi$ be an $L$-parameter for~$G$.
If $\phi$ is discrete, then $\tn{Lev}(\widehat G)^\phi$ is a finite set.
\end{Lemma}

\begin{proof}
Since $\tn{Lev}(\widehat G)^\phi$ is a closed subvariety
of the space of Levi subgroups of~$\widehat G$,
it suffices to show that $\dim(\tn{Lev}(\widehat G)^\phi) = 0$.
Suppose to the contrary that $\tn{Lev}(\widehat G)^\phi$
has a component of positive dimension.
Choose a Levi~$\widehat M$ in this component,
so that the dimension of the tangent space of $\tn{Lev}(\widehat G)^\phi$
at~$\widehat M$ is nonzero.
Since the Lie algebra of $\tn{Lev}(\widehat G)$ at~$\widehat M$
is $\hat{\frak g}/\hat{\frak m}$
and the formation of the Lie algebra commutes with fixed points,
$(\hat{\frak g}/\hat{\frak m})^\phi\neq0$.
Now consider the short exact sequence
\[
\begin{tikzcd}
0 \rar & \hat{\frak m}^\phi \rar & \hat{\frak g}^\phi \rar
& (\hat{\frak g}/\hat{\frak m})^\phi \rar & 0,
\end{tikzcd}
\]
which is exact on the right because $\phi$ is Frobenius semisimple.
Then $\frak z(\hat{\frak g})^\phi\subseteq\frak m^\phi \neq \hat{\frak g}^\phi$,
meaning that $\phi$ is discrete.
\end{proof}

Recall that a twisted Levi subgroup $M$ of~$G$ is \textit{elliptic} (in~$G$)
if $Z^\circ(M)/Z^\circ(G)$ is anisotropic,
or equivalently, if $M$ contains a maximal torus that is elliptic in~$G$.
Since every semisimple reductive $F$-group contains an anisotropic maximal torus
\cite[Theorem~10.5.1]{KalethaPrasad},
when two twisted Levi subgroups of $G$ are stably conjugate,
one of them is elliptic in $G$ if and only if the other one is.
The elliptic twisted Levi subgroups
are the only ones that arise in $\tn{Lev}(\widehat G)^\phi$:

\begin{Corollary}\label{ellipticcor}
Let $\phi$ be an $L$-parameter and
let $M$ be a twisted Levi subgroup such that
$\phi$ factors through ${}^LM_\pm\to{}^LG$.
If $\phi$ is discrete, then $M$ is elliptic in~$G$.
\end{Corollary}

\begin{proof}
This follows from the standard fact that for a reductive group $H$,
the split rank of $Z(H)$ equals the rank of $Z(\widehat{H})^{\Gamma,\circ}$,
applied to $H = G$ and $H =M$ separately.
\end{proof}

Hence every discrete $L$-parameter $\phi$
yields a finite set $\tn{Lev}(\widehat G)^\phi$
of Levi subgroups, which forms a poset under inclusion.
The maximum element is $\widehat G$,
but there may be many minimal elements.

\subsection{The non-basic conjectures for discrete parameters}
Let $\phi$ be a discrete $L$-parameter for~$G$.
Suppose $\phi$ normalizes some Levi subgroup~$\widehat N$,
yielding the quasi-split twisted Levi subgroup $N$ of~$G$.
By Lemma~\ref{lem:lgrouppullback},
the $L$-parameter~$\phi$ factors through~${}^LN_\pm$,
yielding an $L$-parameter $\phi_{N,\pm}$ for $N(F)_\pm$.
We are interested in the compound $L$-packet 
\[
\Pi_\phi^\cE(N_\pm)
\defeq \{\dot{\pi}_{\pm} = (N', z, \pi_{\pm}) \in \Pi_{\phi_{N,\pm}}
\mid [z] \in H^{1}_{\tn{bas}}(\cE, N)_\Greg\}
\subseteq \Pi_{\phi_{N,\pm}}^\cE,
\]
as well as the $W(G,N)(F)$-stable subset $\Pi^\cE_{\phi}(N'_\pm,z)$
of representations of $(N',\psi, z)$.

\begin{Definition}
A \textit{rigid enhancement} of a discrete $L$-parameter~$\phi$
is a pair $(\mc{N},\rho)$ where
\begin{enumerate*}
\item $\mc{N}$ is a $\phi$-stable Levi subgroup of $\widehat G$
and \item $\rho$ is an irreducible representation of
$\pi_0(S^+_{\phi_{N,\pm},(\mc{N})})$
\end{enumerate*}
satisfying the following condition:
\begin{equation} \label{defrigenh}
\parbox{0.6\linewidth}{
The character $\chi$ of $\pi_0(Z(\widehat N)_{(\widehat N)}^{\Gamma,+})$
for which the restriction of $\rho$ is $\chi$-isotypic
is $(G,\phi)$-regular (Definition~\ref{gphiregular}).
}
\end{equation}
Write $\Phi^\cE_\phi$ for the set of rigid enhancements of~$\phi$.
The group $S_\phi$ acts on $\Phi^\cE_\phi$ by conjugation,
using Lemma~\ref{lemunivfunctorial},
and two enhancements are \textit{equivalent}
if they lie in the same orbit. As in Remark \ref{restrem}, the condition \eqref{defrigenh} can be detected by restricting $\chi$ to $\pi_{0}(Z(\widehat{G})^{\Gamma,\circ,+}_{(\widehat{N})})$.
\end{Definition}

Recall from Section~\ref{Lembsubsec} that
a choice of pinning of~$N$ yields a family of $F$-rational automorphisms of $N$
(resp. automorphisms of $^{L}N_{\pm}$) given by acting by elements of $W(G,N)(F)$,
which we denote by $\theta_{w}$ (resp. $\theta_{w}^{\vee}$).
Let $\mathfrak{w}_N$ be a Whittaker datum,
and assume that $\mathfrak{w}_N$ arises from the same pinning
by the usual construction.
Then $\theta_{w}(\mathfrak{w}_{N}) = \mathfrak{w}_{N}$.
Writing $\Phi^\cE_\phi(G,\mc{N})\subseteq\Phi^\cE_\phi$
for those rigid enhancements having first component~$\mc{N}$,
we find a bijection
\begin{equation} \label{transisom}
\Phi^\cE_\phi(G,\mc{N}) \longisoto 
\frac{\Phi^{\cE,\tn{bas}}_{\phi_{N,\pm}}(N_\pm)}{W(G,N)(F)}.
\end{equation}
Strictly speaking, it should be checked
that this bijection is independent of
the choice of representative of~$\phi$
within its equivalence class
and the identification of $\widehat{N}$ with~$\mc{N}$,
and we leave these checks to the reader.

Given a discrete $L$-parameter~$\phi$,
let $H^1_\tn{L-reg}(\cE,G)_\phi$ (resp.~$\ov{Y}_{+,\tn{tor}}(G)_\phi$)
be the union of the sets $H^1_\tn{reg}(\cE,G)_N$
(resp.~$\ov{Y}_{+,\tn{tor}}(G)_N$)
where $N$ ranges over stable conjugacy classes of
elliptic twisted Levi subgroups of~$G$
corresponding to $\phi$-stable Levi subgroups of~$\widehat G$.
We can now state the main result of this paper:

\begin{Theorem}\label{maintheorem2}
Let $G$ be a quasi-split connected reductive group and
$\phi$ a discrete $L$-parameter for~$G$.
Assume that the basic rigid local Langlands conjecture
(Conjecture \ref{rigidLLC1}) holds for $G$
and each of its quasi-split elliptic twisted Levi subgroups.
Then there is a bijection
\begin{equation}\label{rigidLLC2}
\iota^\tn{rig}_\mathfrak{w}\colon
\bigsqcup_{[N]} \frac{\Pi^\cE_\phi(N_\pm)}{W(G,N)(F)}
\longrightarrow \Phi^\cE_\phi(G) 
\end{equation}
where $[N]$ ranges over stable conjugacy classes of
elliptic twisted Levi subgroups of~$G$
corresponding to a $\phi$-stable Levi subgroup of~$\widehat G$.
The bijection fits into a commutative diagram
\begin{equation*}
\begin{tikzcd}%
\displaystyle\bigsqcup_{[N]} \frac{\Pi^\cE_\phi(N_\pm)}{W(G,N)(F)}
\rar[twoheadrightarrow]\dar{\iota^\tn{rig}_\mathfrak{w}}
& H^1_\tn{L-reg}(\cE,G)_\phi \dar{\ref{newtondecomp},\, \ref{TNgphireg}}[swap]{\sim}
\rar[hookrightarrow]{\kappa}
& \displaystyle\ov{Y}_{+,\tn{tor}}(G)_\phi
\dar[hookrightarrow]{\ref{dualKott1}} \\
\Phi^\cE_\phi(G) \rar[twoheadrightarrow]
& \displaystyle\bigsqcup_{[N]}
\frac{X^{*}(\pi_{0}(Z(\widehat{N})^{\Gamma,+}_{(\widehat N)}))_\Gphireg}{W(G,N)(F)}
\rar[hookrightarrow]
& \displaystyle\bigsqcup_{[N]} \frac{X^{*}(\pi_{0}(Z(\widehat{G})^{\Gamma,+}_{(\widehat N)}))}{W(G,N)(F)} 
\end{tikzcd}
\end{equation*}
where the top left horizontal map forgets the representation,
the bottom left horizontal map forms the central character as in \eqref{defrigenh},
and the vertical maps are obtained via maps from the numbered results
cross-referenced in their labels.
\end{Theorem}

\begin{proof}
The commutativity of the left square follows from \eqref{transisom}
and Theorem~\ref{rigidLLC1pm}, which depends on Conjecture \ref{rigidLLC1}.
To check the commutativity of the right square,
we restrict it to the diagram
\[
\begin{tikzcd}%
H^1_\tn{L-reg}(\cE,G)_N \dar[swap]{\sim}
\rar[hookrightarrow]
& \ov{Y}_{+,\tn{tor}}(G)_N \rar[hookrightarrow]
& \wt{Y}_{+,\tn{tor}}(Z(N)\to G) \dar{\sim} \\
\dfrac{X^{*}(\pi_{0}(Z(\widehat{N})^{\Gamma,+}_{(\widehat N)}))_\Gphireg}{W(G,N)(F)}
\arrow[rr,hookrightarrow]
&& \dfrac{X^{*}(\pi_{0}(Z(\widehat{G})^{\Gamma,+}_{(\widehat N)}))}{W(G,N)(F)} 
\end{tikzcd}
\]
This diagram commutes by the functoriality of Proposition~\ref{dualKott1}.
\end{proof}

\begin{Remark}%
Using the map
$\bigsqcup_{[N]} \Pi_\phi^\cE(N_\pm)/W(G,N)(F)
\onto H^{1}_{\tn{L-reg}}(\mc{E}, G)_\phi$,
we can further group elements of $\Phi^\cE_\phi$
by their images in $H^{1}(\mc{E}, G)$
and rewrite Theorem \ref{maintheorem2} as a bijection
\[
\bigsqcup_{[x] \in H^{1}_{\tn{L-reg}}(\mc{E},G)_{\phi}}
\frac{\Pi_{\phi_{N,\pm}}^\cE(G_{[x]}, [x])}{W(G,N)(F)}
\longisoto \Phi_\phi^\cE(G)
\]
where for $[x']$ an arbitrary preimage of $[x]$ in $H^{1}_{\tn{bas}}(\mc{E}, N)$,
the set $\Pi_{\phi_{N,\pm}}(G_{[x]}, [x'])$
of isomorphism classes in $\Pi_\phi^\cE(N_\pm)$
with a representative whose rigidification is cohomologous to $[x']$.
Moreover, every twisted Levi subgroup $N$ of $G$ such that $\phi$
factors through $^{L}N_{\pm} \to \prescript{L}{}G$
arises from a class in $H^{1}_\tn{reg}(\cE, G)_\phi$.
In other words, the map \eqref{rigidLLC2} does not miss
any twisted Levi subgroup of $G$ or inner forms thereof.
\end{Remark}

\section{Examples}\label{examplessec}
\subsection{An irregular cohomology class}

\begin{Lemma} \label{PGL2subgroup}
Let $G=\PGL_2$ over a field~$F$ not of characteristic two.
\begin{enumerate}
\item
There is a subgroup $A$ of~$G$ that is isomorphic to~$\mu_2^2$
and is not contained in any torus of~$G$.

\item
Any two such subgroups~$A$ are $G(\overline F)$-conjugate.

\item
The subgroup $A$ normalizes exactly three maximal tori of~$G$,
those of the form $Z_G^\circ(s)$ for $s\in A$ of order two.

\item
$Z_G(A) = A$.

\item
If $\sqrt{-1}\in F$ then $N_G(A)/A \simeq S_3 \simeq \GL_2(\bbF_2)$
and $N_G(A)$ is generated by the groups $T[4]$
where $T$ is a torus normalized by~$A$.
\end{enumerate}
\end{Lemma}

If $\sqrt{-1}\notin F$ then $N_G(A)/A$
will instead be some form of~$S_3$.
One can further show that $N_G(A)$ is a form of~$S_4$,
generalizing the classical geometry of the tetrahedron
inscribed in the (Riemann) sphere.

\begin{proof}
For the first part, to start with,
note that every maximal torus~$T$ of~$\PGL_2$
is of the form $\Res_{E/F}(\bbG_m)/\bbG_m$,
where $E/F$ is a separable quadratic extension,
and the normalizer of this torus is $T\rtimes\Gal(E/F)$.
The subgroup of $N_G(T)$ generated by the two-torsion element of~$T$
and $\Gal(E/F)$ is the desired copy of $\mu_2^2$ in~$\PGL_2$.
Since every torus of $\PGL_2$ has rank one,
and thus contains a unique point of order two,
which is an $F$-point,
this subgroup is not contained in a torus.

For the second part, a simple matrix calculation shows that the centralizer in $\PGL_2$
of the order-two element $\smat1{}{}{-1}$ is the normalizer of the diagonal maximal torus.
It follows that the assignment $s\mapsto Z_G^\circ(s)$
is a bijection between order-two elements of~$\PGL_2$ and maximal tori of~$\PGL_2$;
the inverse sends a torus to its unique element of order two.
Consequently, there is a unique way to enlarge any order-two element of $\PGL_2(F)$
to a copy of $\mu_2^2$.
The fact that all tori of $\PGL_2$ are $G(\overline F)$-conjugate
implies that all copies of $\mu_2^2$ are $G(\overline F)$-conjugate.

For the third part, it is clear that $A$ normalizes the given tori,
so the problem is to show that no more are normalized.
For this, we may assume that $A$ is generated by
$s=\smat{}11{}$ and $t=\smat1{}{}{-1}$.
The space of maximal tori is isomorphic to the variety $X$
of two-element subsets of~$\bbP^1$
and the conjugation action on tori
becomes the action of $\PGL_2$ on $\bbP^1$
by Möbius transformations.
In this model, $s$ corresponds to the transformation $1/z$
and $t$ to the transformation $-z$.
A short calculation shows that the only points of~$X$
fixed by this group of Möbius transformations
are $\{0,\infty\}$, $\{\pm1\}$ and $\{\pm\sqrt{-1}\}$.

For the fourth part, working again in the setting of the previous paragraph,
$Z_G(A)$ is the centralizer of $s$ in $Z_G(t)$,
which one easily computes to be~$A$.
Take $a\in Z_G^\circ(t)$ of order~$4$, so that $a^2 = t$.
Then $a$ commutes with $t$ and $asa^{-1} = a^2s = ts$,
so that $a$ realizes the permutation $(s,ts)(t)$
of the $3$-element set $\{s, t, ts\}$.
Replacing $s$ by $t$ and $ts$ in turn,
one finds the other two $2$-cycles in~$S_3$.
\end{proof}

\noindent
Now we return to the setting of the paper,
where $F$ is nonarchimedean local of characteristic zero.

\begin{Lemma} \label{nontoralcocycle}
Let $G=\PGL_2$ over~$F$.
Assume $\sqrt{-1}\in F$.
\begin{enumerate}
\item
Any form of $\mu_2^2$ that splits over an unramified extension of~$F$
embeds as a subgroup of~$G$.

\item
There is a cocycle $z\in Z^1_\tn{reg}(\mc{E},G)$
and a form~$A$ of $\mu_2^2$ in~$G$
such that the image of $z|_u$ is~$A$.

\end{enumerate}
\end{Lemma}

\begin{proof}
For the first part,
it suffices to show that for any element $n\in N_G(A)$
there is $g\in \PGL_2(F^\tn{unr})$ such that
$g^{-1}\cdot\tn{Fr}(g) = n$;
then the conjugate $g^{-1} A g$ of~$A$
is defined over~$F$ and has Galois action sending~$\tn{Fr}$
to the image of $n$ in $\Aut(A) = S_3$.
For this, let $\cal G$ be the standard maximal parahoric integral model of~$\PGL_2$,
so that $\cal G(\cal O_F)$ is the elements of $\PGL_2(\cal O)$
whose determinant has even valuation.
Then $\smat1{}{}{-1}$ extends to a copy~$A$ of $\mu_2^2$ contained in~$\cal G$,
and since $\sqrt{-1}\in F$,
the three tori normalized by~$A$ are split.
Hence $N_G(A)$ embeds in $\cal G$.
The existence of $g$ now follows from the facts that $H^1(F^\tn{unr}/F,\cal G)=1$
and that $n$ is a finite-order element of $\cal G(\cal O_F)$.

For the second part, let $E/F$ be an unramified cubic extension of~$F$
and take $A = \Res_{E/F}(\mu_2)/\mu_2$.
Let $\widehat A = \Hom(X^*(A),\bbQ/\bbZ)$,
isomorphic to $\widehat A\simeq \bbF_2^2$ with Frobenius action $\smat1101$.
As \cite[\S~5]{Dillery24} explains,
we may identify $\Hom_F(u,Z)$ with $\widehat A$
and the image of $H^1(\cal E,A)$ in $\Hom_F(u,Z)$
with the kernel of the map from $\widehat A$ to its Galois coinvariants.
Since $\widehat A$ has trivial Galois coinvariants,
every element of $\Hom_F(u,Z)$, in particular,
the natural projection $u\to A$, extends to a cocycle~$z$.
By construction, the image of $z|_u$ is~$A$.
\end{proof}

\subsection{Rigid Newton centralizers that are not elliptic twisted Levi subgroups}
\begin{Example} \label{exnotell}
Let $G = \SL_{3}$, which contains $\GL_{2}$ as a maximal proper Levi subgroup in the usual way. For $K/F$ a quadratic extension we have the anisotropic norm-$1$ torus $S\defeq \tn{Res}_{K/F}(\mathbb{G}_{m})^{(1)} \subseteq \GL_{2} \subseteq \SL_{3}$ and we claim that $Z_{\SL_{3}}(S) = \tn{Res}_{K/F}(\mathbb{G}_{m}) \subseteq \GL_{2} \subseteq \SL_{3}$. Evidently $Z_{\SL_{3}}(S)$ is a twisted Levi subgroup of $G$ and contains $\tn{Res}_{K/F}(\mathbb{G}_{m})$; it is easy to see that, over $\ov{F}$, there are only three such Levi subgroups: $\tn{Res}_{K/F}(\mathbb{G}_{m})$ itself, $\GL_{2}$, or $\SL_{3}$. Since $S$ is not central in $\GL_{2}$, the claim follows. 

We may then pick a cocharacter $\lambda \in X_{*}(\tn{Res}_{K/F}(\mathbb{G}_{m})^{(1)})$ with $Z_{\SL_{3}}(\lambda) = Z_{\SL_{3}}(S) = \tn{Res}_{K/F}(\mathbb{G}_{m})$; since $\lambda$ factors through an anisotropic torus, it is evidently killed by the $\Gamma$-norm, and it follows from the proof of Theorem \ref{twistedLevithm} that $Z_{\SL_{3}}(\lambda) = \tn{Res}_{K/F}(\mathbb{G}_{m})$ is a rigid Newton centralizer. However, it evidently does not contain an elliptic maximal torus of $G$, since it is a non-elliptic maximal torus. 
\end{Example}

In fact, the situation is even more general than the above example suggests: There are rigid Newton centralizers in $G$ which are not twisted Levi subgroups. Indeed, if $T$ is an anisotropic maximal torus of $G$ and $s \in T[n](F)$ is a torsion element such that $Z_{G}(s)$ is not a twisted Levi subgroup, then $Z_{G}(s)$ is such a subgroup, since one can extend the homomorphism $\mu_n \to T$ sending a choice of $n$th root of unity to $s$ to a cocharacter $\lambda\colon \mathbb{G}_{m,\ov{F}} \to T_{\ov{F}}$ (via the canonical inclusion $\mu_n \hookrightarrow T$). The fact that $T$ is anisotropic guarantees that $\lambda$ is killed by $N_{E/F}$, and hence by Proposition \ref{Newtonimageprop} the group $Z_{G}(\lambda|_{\mu_n}) = Z_{G}(s)$ is defined over $F$ and is thus
regular. The following example gives such an element $s$.

\begin{Example}\label{nonadmex}
Let $G = G_{2}$ with a choice of $T$ an anisotropic maximal torus split over a quadratic extension $E/F$ on which $\Gamma_{E/F}$ acts by inversion---in particular, we have $T[2](\ov{F}) = T[2](F)$. Let $a,b$ be the short and long (respectively) elements of a root basis corresponding to a choice of Borel subgroup containing $T_{\ov{F}}$, with corresponding fundamental coweights $u,v$ (respectively); the claim is that if we set $s = v(-1) \in T[2](F)$ then $Z_{G_{2}}(s)$ is not a twisted Levi subgroup of $G_{2}$.

For any root $\alpha \in R(G_{2,\ov{F}}, T_{\ov{F}})$, one checks that $\alpha(s) = (-1)^{n}$, where $n$ is the coefficient modulo $2$ of $a$ in $\alpha$. The only such $\alpha$ which are positive and have trivial $a$-coefficients modulo $2$ are $b$ or $2a + b$, which span a root system of type $A_{1} \times A_{1}$, and so we deduce that $Z_{G_{2}}(s)_{\ov{F}}$ is isomorphic to $\tn{SO}_{4}$, giving the claim.
\end{Example}

\subsection{A strongly regular depth-zero $L$-parameter normalizing two tori}

\begin{Lemma} \label{G2subgroup}
Let $G=G_2$ over an algebraically closed field~$F$ not of characteristic two.
\begin{enumerate}
\item
There is a unique $G(F)$-conjugacy class of subgroups~$A$ isomorphic to~$\mu_2^3$.

\item
$Z_G(A) = A$ and $N_G(A)/A\simeq\Aut(A)\simeq\GL_3(\bbF_2)$.

\item
There is a unique $N_G(A)$-conjugacy class of subgroups~$B$ of~$N_G(A)$
that contain~$A$ and have order~$168$.

\item
The group~$B$ does not normalize a maximal torus of~$G$.
\end{enumerate}
\end{Lemma}

\begin{proof}
For the first part, any subgroup isomorphic to $\mu_2^2$
is contained in a maximal torus~$T$ \cite[Theorem~2.27]{Steinberg75}.
Since $N_G(T) \simeq T \rtimes W(G,T)$ \cite[Proposition 3.17]{Adrian22},
a lift $s$ of $-1\in W(G,T)$ of order two
together with the group $T[2]\simeq\mu_2^2$
generates a subgroup isomorphic to~$\mu_2^3$.

For the second part,
the description of $A$ given above shows that $Z_G(A) = A$.
It follows that $N_G(A)/A\subseteq\Aut(A)$.
To show that this inclusion is an equality,
one can use the interpretation of~$G_2$
as the automorphism group of the octonions and describe $N_G(A)$
as the automorphisms of the multiplication table of the octonions \cite{Coxeter46}.

For the third part, since $\GL_3(\bbF_2)$ has order~$168 = 8\cdot 3\cdot 7$,
it contains a unique conjugacy class of subgroups of order~$7$,
the Sylow $7$-subgroups, which can be identified with the maximal tori $\bbF_8^\times$.
Using the general description of normalizers of elliptic tori in~$\GL_n$
one sees that the normalizer of an $\bbF_8^\times$
is a semidirect product $\bbF_8^\times\rtimes\Gal(\bbF_8/\bbF_2)$, of order~$3\cdot 7$.

For the fourth part, suppose $B$ normalizes a maximal torus~$T$ of~$G$
and let $C$ be a Sylow $7$-subgroup of~$B$.
Since $W(G,T)$ is the dihedral group of order~$12$,
the image of~$C$ in $W(G,T)$ is trivial,
so that $C\subseteq T$.
The same argument shows that there is a subgroup $A'\subseteq A$,
of order at least two, such that $A'\subseteq T$, as $8\nmid12$.
Hence some nontrivial subgroup of~$A$ commutes with~$C$.
At the same time, recalling from the third part
that the conjugation action of~$C$ on~$A$
can be identified with the multiplication action of $\bbF_8^\times$ on~$\bbF_8$,
we see that $C$ acts simply transitively on $A\setminus\{1\}$, a contradiction.
\end{proof}

Let $G=G_2$ and let $\widehat T\subseteq\widehat G=G_2(\bbC)$ be a maximal torus.
Suppose that $q\equiv -1\bmod 6$.
The Weyl group of $G_2$ is the dihedral group of order~$12$,
generated by a reflection and a rotation of order~$6$.
It turns out \cite{Adrian22,MO476119}
that the Weyl group also lifts to~$G_2$,
meaning that there is a homomorphic section
$W(\widehat G,\widehat T)\to N_{\widehat G}(\widehat T)$
of the natural projection.
Let $W\subseteq G_2(\bbC)$ denote the image of some such section.
Define the $L$-parameter $\varphi\colon W_F\to G_2$
to be trivial on wild inertia,
send a generator $s$ of tame inertia to an order-$6$ rotation in~$W$,
and send a Frobenius element $f$ to a reflection in~$W$.
Identify $s$ and $f$ with their images in~$G_2$.
We claim that
\begin{enumerate}
\item
$\varphi$ is strongly regular,
meaning in this case that $\varphi$ is discrete
and $s$ is strongly regular, and

\item
$\varphi$ normalizes two tori, $\widehat T$ and $Z_{\widehat G}(s)$.

\end{enumerate}
Since $G_2$ is simply connected strongly regular is the same as regular
\cite[2.15]{Steinberg75}
and it suffices to show that $s$ is regular.
For this, we use \cite[Lemma~5.2]{Reeder10},
which applies because $s$ is elliptic for~$\widehat T$
and has two orbits on the root system,
the long roots and the short roots.
Hence $s$ is regular.
Since $fsf = s^{-1}$ and $Z_{\widehat G}(s) = Z_{\widehat G}(s^{-1})$,
the torus $\widehat S = Z_{\widehat G}(s)$ is also normalized by~$\varphi$,
proving the second claim.
For the first claim, it remains to show that $\varphi$ is discrete.
Here it is easier to work with $\widehat S$.
It suffices to show that the image of $f$ in $W(\widehat G,\widehat S)$
is the elliptic element $-1$.
Indeed, the element $-1$ takes $s$ to~$s^{-1}$
and since $s$ is strongly regular there is a unique
such element of the Weyl group.

\subsection{Functoriality and the basic Kottwitz map}\label{kottfuncex}
Recall from the work of Kaletha that there is a bijection
$\kappa\colon H^{1}_\tn{bas}(\cE, G) \isoto \ov{Y}_{+,\text{tor}}(Z(G)\to G)$.
In analogy with the Kottwitz map $B(G)\to\pi_1(G)_\Gamma$ for the isocrystal gerbe,
which is bijective on the basic subset, one can ask if there is a map
\begin{equation*}
H^{1}_\tn{reg}(\cE, G) \to \ov{Y}_{+,\text{tor}}(Z(G)\to G)
\end{equation*}
which
\begin{enumerate*}
\item restricts to $\kappa$ on the basic subset and
\item is functorial in $G$.
\end{enumerate*}
Since $\ov{Y}_{+,\text{tor}}(-)$ is not obviously functorial,
it is difficult to imagine how this could be possible.
Here we answer the question in the negative by giving
an explicit example where these two conditions cannot both hold.

 Let $G = \mathrm{SL}_{2}$ with elliptic maximal torus $T$ the norm-one torus $\mathrm{Res}_{E/F}(\mathbb{G}_{m})^{(1)}$ for a quadratic extension $E/F$. One calculates that $H^{1}_\tn{bas}(\cE, \mathrm{SL}_{2}) \isoto \mathbb{Z}/2\mathbb{Z} \isoto \ov{Y}_{+,\text{tor}}(\mu_2\to\mathrm{SL}_{2})$, and we also know from \cite[Corollary 3.7]{Kaletha16a} that every class in $H^{1}_\tn{bas}(\cE, \mathrm{SL}_{2})$ is the image of some class in $H^{1}(\cE, T)$.

Note that $H^{1}(\cE, T) \isoto \ov{Y}_{+,\text{tor}}(T) \isoto X^*(\varprojlim \mu_{2n}(\mathbb{C})) = \bbQ/\bbZ$. If $\ov{Y}_{+,\text{tor}}(Z(G)\to G)$ were functorial in $G$,
then the map $\ov{Y}_{+,\text{tor}}(T) \to \ov{Y}_{+,\text{tor}}(\mu_2\to\mathrm{SL}_{2})$ would be a group homomorphism $\mathbb{Q}/\mathbb{Z} \to \mathbb{Z}/2\mathbb{Z}$, which can only be trivial since the source is divisible and the target is torsion. 
This observation means that the diagram 
\[
\begin{tikzcd}
H^{1}(\cE, T) \arrow{r} \dar{\sim} & H^{1}_\tn{reg}(\cE, \mathrm{SL}_{2})
\dar[twoheadrightarrow] \\
\ov{Y}_{+,\text{tor}}(T) \rar{0} & \ov{Y}_{+,\text{tor}}(\mu_2\to\mathrm{SL}_{2})
\end{tikzcd}
\]
cannot commute and be such that the induced map $H^{1}_\tn{bas}(\cE,\mathrm{SL}_{2}) \to \ov{Y}_{+,\text{tor}}(\mu_2 \to \mathrm{SL}_{2})$ recovers $\kappa$. Indeed, if this were the case then the right-down composition would be nontrivial but the down-right composition is trivial, which is a contradiction.

\printbibliography

\end{document}